\documentclass[12pt]{amsart}

\usepackage{amsfonts, amssymb, latexsym, epsfig, amsthm}
\usepackage[usenames,dvipsnames]{xcolor}
\usepackage{graphicx,pgf,tikz}
\usepackage{ytableau}
\usepackage[all]{xy}
\usepackage{amsmath}
\usepackage{enumerate}
\usepackage{array}

\theoremstyle{plain}
\newtheorem{theorem}{Theorem}[section]

\newtheorem{proposition}[theorem]{Proposition}

\newtheorem{lemma}[theorem]{Lemma}

\newtheorem{corollary}[theorem]{Corollary}
\newtheorem{conjecture}[theorem]{Conjecture}

\newtheorem{claim}[theorem]{Claim}
\theoremstyle{definition}

\newtheorem{example}[theorem]{Example}

\newcommand{\excise}[1]{}

\newcommand{\rdots}{\mathinner{%
  \mkern1mu\raise1pt\hbox{.}%
  \mkern2mu\raise4pt\hbox{.}%
  \mkern2mu\raise7pt\vbox{\kern7pt\hbox{.}}\mkern1mu}}

\numberwithin{equation}{section}

\newcommand{\plusr}{{\color{Rhodamine}+}}
\newcommand{\plusg}{{\color{red}+}}
\newcommand{\plusb}{{\color{blue}+}}

\newcommand{\pluso}{{\color{orange}+}}
\newcommand{\plusy}{{\color{ForestGreen}+}}

\begin{document}
\pagestyle{plain}
\title{The Prism tableau model for Schubert polynomials}
\author{Anna Weigandt}
\author{Alexander Yong}
\address{Dept.~of Mathematics, U.~Illinois at
Urbana-Champaign, Urbana, IL 61801, USA}
\email{weigndt2@uiuc.edu, ayong@uiuc.edu}


\date{September 8, 2015}

\maketitle 

\begin{abstract}
The  Schubert polynomials 
lift the Schur basis of symmetric polynomials into
a basis for ${\mathbb Z}[x_1,x_2,\ldots]$. 
We suggest the \emph{prism tableau model} for these polynomials.
A novel aspect of this alternative to earlier results is that it 
directly invokes semistandard tableaux; it does so as part of a colored
tableau amalgam. In the Grassmannian case, a prism tableau with colors ignored is a semistandard Young tableau. Our arguments are developed from the Gr\"{o}bner geometry of
matrix Schubert varieties.  
\end{abstract}

\tableofcontents

\section{Introduction}

\subsection{Overview}

A.~Lascoux--M.-P.~Sch\"{u}tzenberger \cite{Lascoux.Schutzenberger} 
recursively defined an integral basis of ${\sf Pol}={\mathbb Z}[x_1,x_2,\ldots]$ given by 
the {\bf Schubert polynomials} $\{{\mathfrak S}_w:w\in S_\infty\}$. If $w_0$ is the longest length permutation 
in the symmetric group $S_n$ then ${\mathfrak S}_{w_0}:=x_1^{n-1}x_2^{n-2}\cdots x_{n-1}$. Otherwise, $w\neq w_0$ and 
there exists $i$ such that $w(i)<w(i+1)$. Now one sets
${\mathfrak S_w}=\partial_i {\mathfrak S}_{ws_i}$, where
$\partial_i f:= \frac{f-s_if}{x_i-x_{i+1}}$  
(since the polynomial operators $\partial_i$ form a representation of $S_n$, this definition is self-consistent.) It is true that
under the standard inclusion $\iota:S_n\hookrightarrow S_{n+1}$, ${\mathfrak S}_w={\mathfrak S}_{\iota(w)}$. Thus one can
refer to ${\mathfrak S}_w$ for each $w\in S_{\infty}=\bigcup_{n\geq 1} S_n$. 

Textbook understanding of the ring {\sf Sym} of symmetric polynomials centers
around the basis of Schur polynomials 
and its successful companion, the theory of Young tableaux.  
Since Schur polynomials are instances of Schubert polynomials, the 
latter basis naturally lifts the Schur basis into {\sf Pol}. Yet, it is also true that
Schubert polynomials have nonnegative integer 
coefficients. Consequently, one has a natural problem:
\begin{quotation}
Is there a combinatorial model for Schubert polynomials that 
is analogous to the
semistandard tableau model for Schur polynomials?
\end{quotation}
Indeed, multiple solutions have been discovered over the years, e.g., \cite{Kohnert}, \cite{BJS},
\cite{Bergeron.Billey}, \cite{Fomin.Stanley}, \cite{Fomin.Kirillov}, \cite{balanced}, \cite{Magyar}, \cite{Bergeron.Sottile, Bergeron.Sottile:II},
\cite{BKTY} and \cite{Taskin} (see also \cite{LS:85}). In turn, the solutions \cite{BJS, Bergeron.Billey, Fomin.Stanley, Fomin.Kirillov} have been the foundation
for a vast literature at the confluence of combinatorics, representation theory and combinatorial algebraic geometry.

We wish to put forward another solution -- a novel aspect of which is that it directly invokes semistandard tableaux. Both the statement and proof of our alternative model build upon ideas 
about the Gr\"{o}bner geometry of matrix Schubert varieties $X_w$. 
 We use the 
Gr\"obner degeneration of $X_w$ and the interpretation
of ${\mathfrak S}_w$ as mutidegrees of $X_w$ \cite{Knutson.Miller:annals}. Actually,
a major purpose of \emph{loc. cit.} is to establish the geometric naturality of the combinatorics of
\cite{BJS, Bergeron.Billey, Fomin.Kirillov}. Our point of departure is stimulated by later work
of A.~Knutson on \emph{Frobenius splitting} \cite[Theorem~6 and Section~7.2]{Knutson:Frob}.

\subsection{The main result}
We recall some permutation combinatorics found in, e.g., in \cite{Manivel}. 
The {\bf diagram of $w$} is $D(w)=\{(i,j): 1\leq i,j\leq n, w(i)>j \text{\ and \ } w^{-1}(j)>i\}\subset n\times n$. 
Let ${\mathcal Ess}(w)\subset D(w)$ be the {\bf essential set} of $w$: the southeast-most boxes of each connected component of $w$. The {\bf rank function} is 
$r_w(i,j)=\#\{t\leq i:w(t)\leq j\}$.

Define $w$ to be {\bf Grassmannian} if it has at most one descent, i.e., at most one index $k$ such that $w(k)>w(k+1)$. If in addition 
$w^{-1}$ is Grassmannian then $w$ is {\bf biGrassmannian}.
For $e=(i,j)\in {\mathcal Ess}(w)$ let $R_e$ be the $(i-r_w(i,j))\times (j-r_w(i,j))$ rectangle 
with southwest corner at position $(i,1)$ of $n\times n$.  The {\bf shape} 
of $w$ is $\lambda(w)=\bigcup_{e\in\mathcal Ess(w)}R_e$:

\begin{figure}[h]
\label{fig:first}
\begin{picture}(400,70)
\put(30,0){\includegraphics[scale=.61]{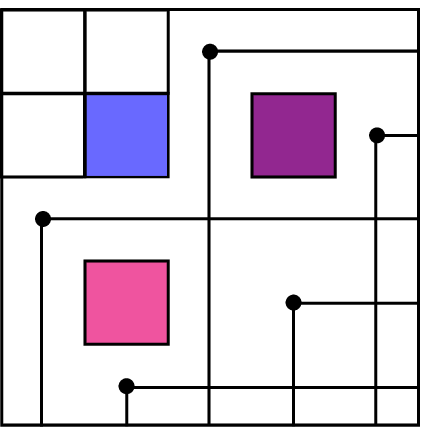}}
\put(168,41){\includegraphics[scale=.19]{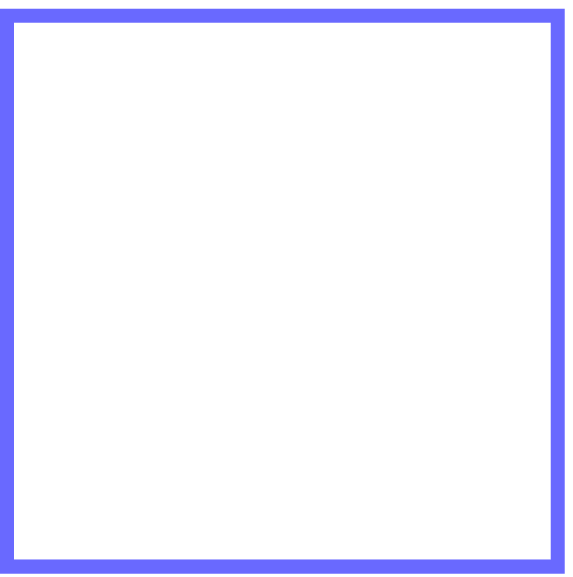}}
\put(168.5,9){\includegraphics[scale=.19]{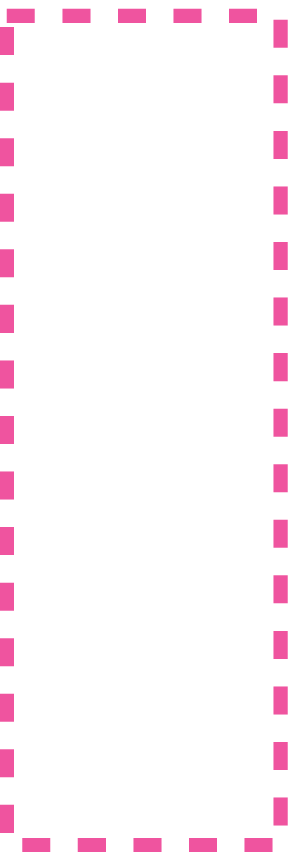}}
\put(169,40.5){\includegraphics[scale=.19]{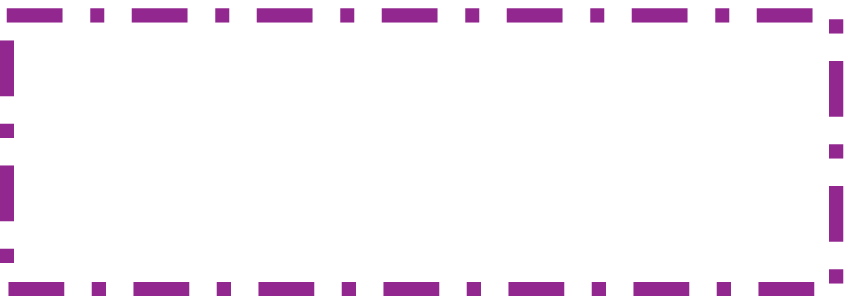}}

\put(48,50){$e_1$}
\put(78,50){$e_2$}
\put(47,20){$e_3$}
\put(195,30){$R_{e_1}$}
\put(215,30){$R_{e_2}$}
\put(187,5){$R_{e_3}$}
\put(285,42){$\lambda(w)=$}
\put(132,40){$\Rightarrow$}
\put(249,40){$\Rightarrow$}
\put(320,58){
\ytableausetup{baseline}
\resizebox{48pt}{!}{
\begin{ytableau}
\ &\ \\
\ \ \ & \ &\ \\
\ \\
\ 
\end{ytableau}}}
\end{picture}
\caption{The diagram of $w=35142$ (with color coded essential set $\{e_1,e_2,e_3\}$),
the overlay of $R_{e_1}, R_{e_2}, R_{e_3}$, and the shape $\lambda(w)$.}
\end{figure}

A {\bf prism tableau} $T$ for $w$ fills $\lambda(w)$ with colored labels (one color for each $e\in {\mathcal Ess}(w)$) such that the labels of color $e$:
\begin{itemize}
\item[(S1)] sit in a box of $R_e$;
\item[(S2)] weakly decrease along rows from left to right;
\item[(S3)] strictly increase along columns from top to bottom; and
\item[(S4)] are {\bf flagged}: a label is no bigger than the row of the box it
sits in.
\end{itemize}
Let $d_i(w)$ be the number of distinct values (ignoring color) seen on the $i$-th 
antidiagonal (i.e., the one meeting $(i,1)$), for $i=1,2,\ldots,n$. 
We say $T$ is {\bf minimal} if $\sum_{i=1}^n d_i(w)=\ell(w)$, where $\ell(w)$ is the \emph{Coxeter length} of $w$. 

Let $\ell_c$ be a label $\ell$ of color $c$. 
Labels $\{\ell_c,\ell_d, \ell'_e\}$
in the same antidiagonal form an {\bf unstable triple}
if $\ell<\ell'$ and replacing the $\ell_c$ with $\ell'_c$
gives a prism tableau. See Example~\ref{exa:unstable}.
Let ${\tt Prism}(w)$ be the set of minimal prism tableaux with no unstable triples.
Finally, set \[{\mathfrak P}_{w}(x_1,\ldots,x_n):=\sum_{T\in {\tt Prism}(w)} {\tt wt}(T), 
\text{\ \ where ${\tt wt}(T)=\prod_{i}x_i^{\text{$\#$ of antidiagonals containing $i$}}$}.\] 

\begin{theorem}
\label{theorem:main}
${\mathfrak S}_w(x_1,\ldots,x_n)={\mathfrak P}_w(x_1,\ldots,x_n)$.
\end{theorem}

\begin{example}[Reduction to semistandard tableaux]
Consider the Grassmannian permutation $w=246135$.  Conflating prism tableaux with their weights, 
Theorem~\ref{theorem:main} asserts:
\begin{center}
\ytableausetup{ boxsize=1.1em,aligntableaux=center,nobaseline}
$\mathfrak S_w=\begin{ytableau}\color{blue} 1\\ \color{blue} 2\color{Plum}
       2&\color{Plum} 2\\  \resizebox{17pt}{!}{ \color{blue} 3\color{Plum} 3\color{Rhodamine}
       3 }&\color{Plum} 3\color{Rhodamine} 3&\color{Rhodamine} 1\\
       \end{ytableau} + \begin{ytableau}\color{blue} 1\\ \color{blue}
       2\color{Plum} 2&\color{Plum} 2\\  \resizebox{17pt}{!}{ \color{blue}
       3\color{Plum} 3\color{Rhodamine} 3 }&\color{Plum} 3\color{Rhodamine}
       3&\color{Rhodamine} 2\\ \end{ytableau} + \begin{ytableau}\color{blue} 1\\
       \color{blue} 2\color{Plum} 2&\color{Plum} 2\\ \resizebox{17pt}{!}{ \color{blue}
       3\color{Plum} 3\color{Rhodamine} 3 }&\color{Plum} 3\color{Rhodamine}
       3&\color{Rhodamine} 3\\ \end{ytableau} + \begin{ytableau}\color{blue} 1\\
       \color{blue} 2\color{Plum} 2&\color{Plum} 1\\  \resizebox{17pt}{!}{ \color{blue}
       3\color{Plum} 3\color{Rhodamine} 3 }&\color{Plum} 2\color{Rhodamine}
       2&\color{Rhodamine} 1\\ \end{ytableau} + \begin{ytableau}\color{blue} 1\\
       \color{blue} 2\color{Plum} 2&\color{Plum} 1\\ \resizebox{17pt}{!}{ \color{blue}
       3\color{Plum} 3\color{Rhodamine} 3 }&\color{Plum} 2\color{Rhodamine}
       2&\color{Rhodamine} 2\\ \end{ytableau} + \begin{ytableau}\color{blue} 1\\
       \color{blue} 2\color{Plum} 2&\color{Plum} 1\\  \resizebox{17pt}{!}{ \color{blue}
       3\color{Plum} 3\color{Rhodamine} 3 }&\color{Plum} 3\color{Rhodamine}
       3&\color{Rhodamine} 1\\ \end{ytableau} + \begin{ytableau}\color{blue} 1\\
       \color{blue} 2\color{Plum} 2&\color{Plum} 1\\ \resizebox{17pt}{!}{ \color{blue}
       3\color{Plum} 3\color{Rhodamine} 3 }&\color{Plum} 3\color{Rhodamine}
       3&\color{Rhodamine} 2\\ \end{ytableau} + \begin{ytableau}\color{blue} 1\\
       \color{blue} 2\color{Plum} 2&\color{Plum} 1\\  \resizebox{17pt}{!}{ \color{blue}
       3\color{Plum} 3\color{Rhodamine} 3 }&\color{Plum} 3\color{Rhodamine}
       3&\color{Rhodamine} 3\\ \end{ytableau}.$
\end{center}
Forgetting colors gives the following expansion of the Schur polynomial:
\begin{center}
$s_{\lambda(w)}=\begin{ytableau} 1\\2&2\\3&3&1\end{ytableau}+
\begin{ytableau} 1\\2&2\\3&3&2\end{ytableau}+
\begin{ytableau} 1\\2&2\\3&3&3\end{ytableau}+
\begin{ytableau} 1\\2&1\\3&2&1\end{ytableau}+
\begin{ytableau} 1\\2&1\\3&2&2\end{ytableau}+
\begin{ytableau} 1\\2&1\\3&3&1\end{ytableau}+
\begin{ytableau} 1\\2&1\\3&3&2\end{ytableau}+
\begin{ytableau} 1\\2&1\\3&3&3\end{ytableau}$.
\end{center}
In general, if $w$ is Grassmannian 
then $\lambda(w)$ is a (French) Young diagram. Moreover,
each cell of $T\in {\tt Prism}(w)$ uses only one number. (See Lemma~\ref{prop:grassredux}.)
Replacing each set in $T$ by the common value gives a \emph{reverse} semistandard tableau.
Thus ${\mathfrak P}_w=s_{\lambda}(w)$ follows. \qed
\end{example}

Prism tableaux provide a means to understand the $RC$-graphs of \cite{Bergeron.Billey, Fomin.Kirillov}.
We think of the  $\#{\mathcal Ess}(w)$-many semistandard tableaux
of a prism tableau $T$ as the ``dispersion'' of the associated $RC$-graph
through $T$. See Sections~4.1 and~4.3.

Minimality and the unstable triple condition bond the tableau of each color,
which is one reason why we prefer not to think of a prism tableau as merely a $\#{\mathcal Ess}(w)$-tuple:

\begin{example}[Unstable triples]
\label{exa:unstable}
Let $w=42513$.  Then $\#{\mathcal Ess}(w)=3$.  The minimal prism tableaux and their weights are:

\begin{center}
\begin{tabular}{|c||c|c|c|c|}\hline
\rule{0pt}{5ex}\rule[-4ex]{0pt}{0pt}  $T$ &
\ytableausetup{baseline}
\begin{ytableau}
{\color{Plum}1} {\color{blue}1}&{\color{blue}1}&{\color{blue}1}\\
{\color{Plum}2}{\color{Rhodamine}2}&{\color{Rhodamine}1}\\
{\color{Plum}3}{\color{Rhodamine}3}&{\color{Rhodamine}3}
\end{ytableau} &
\ytableausetup{baseline}
\begin{ytableau}
{\color{Plum}1} {\color{blue}1}&{\color{blue}1}&{\color{blue}1}\\
{\color{Plum}2}{\color{Rhodamine}1}&{\color{Rhodamine}1}\\
{\color{Plum}3}{\color{Rhodamine}3}&{\color{Rhodamine}3}
\end{ytableau}&
\ytableausetup{baseline}
\begin{ytableau}
{\color{Plum}1} {\color{blue}1}&{\color{blue}1}&{\color{blue}1}\\
{\color{Plum}2}{\color{Rhodamine}2}&{\color{Rhodamine}1}\\
{\color{Plum}3}{\color{Rhodamine}3}&{\color{Rhodamine}2}
\end{ytableau}&
\ytableausetup{baseline}
\begin{ytableau}
{\color{Plum}1} {\color{blue}1}&{\color{blue}1}&{\color{blue}1}\\
{\color{Plum}2}{\color{Rhodamine}1}&{\color{Rhodamine}1}\\
{\color{Plum}3}{\color{Rhodamine}3}&{\color{Rhodamine}2}
\end{ytableau}\\ \hline
\rule{0pt}{2.6ex}\rule[-1.2ex]{0pt}{0pt} ${\tt wt}(T)$ &
$x_1^3x_2x_3^2$ &$x_1^3x_2x_3^2$ & $x_1^3x_2^2x_3$ & $x_1^3x_2^2x_3$\\\hline
\end{tabular}
\end{center}
The second and the fourth tableaux have an unstably paired label. In both tableaux, the pink $1$ in the second antidiagonal is replaceable by a pink $2$.  So  ${\mathfrak S}_w=x_1^3x_2x_3^2+x_1^3x_2^2x_3$.\qed
\end{example}

\ 

\begin{table}[!h]
\label{table:S4}
\ytableausetup{smalltableaux,aligntableaux=center,nobaseline}
\begin{tabular}{|p{.85cm}|p{5.3cm}|p{3.4cm}||p{1cm}|p{2.6cm}|p{2cm}|}
      \hline \begin{center}$\mathfrak S_{1234}$\end{center} & \begin{center}$\emptyset$\end{center} & \begin{center} $1$ \end{center} & \begin{center}$\mathfrak S_{3124}$\end{center}&\begin{center}\begin{ytableau}\color{blue} 1&\color{blue} 1\\ \end{ytableau} \end{center}&\begin{center}$x_1^2$ \end{center}\\ 
      \hline \begin{center}$\mathfrak S_{1243}$\end{center}& \begin{center}\begin{ytableau}\color{blue} 1\\ \end{ytableau} + \begin{ytableau}\color{blue} 2\\ \end{ytableau} + \begin{ytableau}\color{blue} 3\\ \end{ytableau}  \end{center}&\begin{center}$x_1 + x_2 + x_3$ \end{center} & \begin{center}$\mathfrak S_{3142}$\end{center}&\begin{center}\begin{ytableau}\color{blue} 1&\color{blue} 1\\ \color{Plum} 1\\ \color{Plum} 2\\ \end{ytableau} + \begin{ytableau}\color{blue} 1&\color{blue} 1\\ \color{Plum} 1\\ \color{Plum} 3\\ \end{ytableau} \end{center}&\begin{center}$x_1^2x_2 + x_1^2x_3$ \end{center}\\ 
      \hline \begin{center}$\mathfrak S_{1324}$\end{center}&\begin{center}\begin{ytableau}\color{blue} 1\\ \end{ytableau} + \begin{ytableau}\color{blue} 2\\ \end{ytableau} \end{center}&\begin{center}$x_1 + x_2$ \end{center} & \begin{center}$\mathfrak S_{3214}$\end{center}&\begin{center}\begin{ytableau}\color{blue} 1\color{Plum} 1&\color{blue} 1\\ \color{Plum} 2\\ \end{ytableau} \end{center}&\begin{center}$x_1^2x_2$ \end{center}\\ 
      \hline \begin{center}$\mathfrak S_{1342}$\end{center}&\begin{center}\begin{ytableau}\color{blue} 2\\ \color{blue} 3\\ \end{ytableau} + \begin{ytableau}\color{blue} 1\\ \color{blue} 2\\ \end{ytableau} + \begin{ytableau}\color{blue} 1\\ \color{blue} 3\\ \end{ytableau} \end{center}&\begin{center}$x_2x_3 + x_1x_2 + x_1x_3$ \end{center}
& \begin{center}$\mathfrak S_{3241}$\end{center}&\begin{center}\begin{ytableau}\color{blue} 1\color{Plum} 1&\color{blue} 1\\ \color{Plum} 2\\ \color{Plum} 3\\ \end{ytableau} \end{center}&\begin{center}$x_1^2x_2x_3$ \end{center}\\ 
      \hline \begin{center}$\mathfrak S_{1423}$\end{center}&\begin{center}\begin{ytableau}\color{blue} 2&\color{blue} 2\\ \end{ytableau} + \begin{ytableau}\color{blue} 1&\color{blue} 1\\ \end{ytableau} + \begin{ytableau}\color{blue} 2&\color{blue} 1\\ \end{ytableau} \end{center}&\begin{center}$x_2^2 + x_1^2 + x_1x_2$ \end{center} &
\begin{center}$\mathfrak S_{3412}$\end{center}&\begin{center}\begin{ytableau}\color{blue} 1&\color{blue} 1\\ \color{blue} 2&\color{blue} 2\\ \end{ytableau} \end{center}&\begin{center}$x_1^2x_2^2$ \end{center}\\ 
      \hline \begin{center}$\mathfrak S_{1432}$\end{center}&\begin{center}\begin{ytableau}\color{blue} 1\color{Plum} 1&\color{blue} 1\\ \color{Plum} 2\\ \end{ytableau} + \begin{ytableau}\color{blue} 1\color{Plum} 1&\color{blue} 1\\ \color{Plum} 3\\ \end{ytableau} + \begin{ytableau}\color{blue} 2\color{Plum} 1&\color{blue} 2\\ \color{Plum} 2\\ \end{ytableau} + \begin{ytableau}\color{blue} 2\color{Plum} 2&\color{blue} 1\\ \color{Plum} 3\\ \end{ytableau} + \begin{ytableau}\color{blue} 2\color{Plum} 2&\color{blue} 2\\ \color{Plum} 3\\ \end{ytableau} \end{center}&\begin{center}$x_1^2x_2 + x_1^2x_3 + x_1x_2^2 + x_1x_2x_3 + x_2^2x_3$ \end{center} &  \begin{center}$\mathfrak S_{3421}$\end{center}&\begin{center}\begin{ytableau}\color{blue} 1\color{Plum} 1&\color{blue} 1\\ \color{blue} 2\color{Plum} 2&\color{blue} 2\\ \color{Plum} 3\\ \end{ytableau} \end{center}&\begin{center}$x_1^2x_2^2x_3$ \end{center}\\ 
      \hline \begin{center}$\mathfrak S_{2134}$\end{center}&\begin{center}\begin{ytableau}\color{blue} 1\\ \end{ytableau} \end{center}&\begin{center}$x_1$ \end{center} &
\begin{center}$\mathfrak S_{4123}$\end{center}&\begin{center}\begin{ytableau}\color{blue} 1&\color{blue} 1&\color{blue} 1\\ \end{ytableau} \end{center}&\begin{center}$x_1^3$ \end{center}\\ 
      \hline \begin{center}$\mathfrak S_{2143}$\end{center}&\begin{center} \begin{ytableau}\color{blue} 1\\ \color{Plum} 1\\ \end{ytableau} + \begin{ytableau}\color{blue} 1\\ \color{Plum} 2\\ \end{ytableau} + \begin{ytableau}\color{blue} 1\\ \color{Plum} 3\\ \end{ytableau} \end{center}&\begin{center}$x_1^2 + x_1x_2+x_1 x_3$ \end{center} & \begin{center}$\mathfrak S_{4132}$\end{center}&\begin{center}\begin{ytableau}\color{blue} 1&\color{blue} 1&\color{blue} 1\\ \color{Plum} 1\\ \color{Plum} 2\\ \end{ytableau} + \begin{ytableau}\color{blue} 1&\color{blue} 1&\color{blue} 1\\ \color{Plum} 1\\ \color{Plum} 3\\ \end{ytableau} \end{center}&\begin{center}$x_1^3x_2 + x_1^3x_3$ \end{center}\\ 
      \hline \begin{center}$\mathfrak S_{2314}$\end{center}&\begin{center}\begin{ytableau}\color{blue} 1\\ \color{blue} 2\\ \end{ytableau} \end{center}&\begin{center}$x_1x_2$ \end{center} & \begin{center}$\mathfrak S_{4213}$\end{center}&\begin{center}\begin{ytableau}\color{blue} 1\color{Plum} 1&\color{blue} 1&\color{blue} 1\\ \color{Plum} 2\\ \end{ytableau} \end{center}&\begin{center}$x_1^3x_2$ \end{center}\\ 
      \hline \begin{center}$\mathfrak S_{2341}$\end{center}&\begin{center}\begin{ytableau}\color{blue} 1\\ \color{blue} 2\\ \color{blue} 3\\ \end{ytableau} \end{center}&\begin{center}$x_1x_2x_3$ \end{center}& \begin{center}$\mathfrak S_{4231}$\end{center}&\begin{center}\begin{ytableau}\color{blue} 1\color{Plum} 1&\color{blue} 1&\color{blue} 1\\ \color{Plum} 2\\ \color{Plum} 3\\ \end{ytableau} \end{center}&\begin{center}$x_1^3x_2x_3$ \end{center}\\ 
      \hline \begin{center}$\mathfrak S_{2413}$\end{center}&\begin{center}\begin{ytableau}\color{blue} 1\\ \color{blue} 2\color{Plum} 2&\color{Plum} 2\\ \end{ytableau} + \begin{ytableau}\color{blue} 1\\ \color{blue} 2\color{Plum} 2&\color{Plum} 1\\ \end{ytableau} \end{center}&\begin{center}$x_1x_2^2 + x_1^2x_2$ \end{center}& \begin{center}$\mathfrak S_{4312}$\end{center}&\begin{center}\begin{ytableau}\color{blue} 1\color{Plum} 1&\color{blue} 1\color{Plum} 1&\color{blue} 1\\ \color{Plum} 2&\color{Plum} 2\\ \end{ytableau} \end{center}&\begin{center}$x_1^3x_2^2$ \end{center}\\ 
      \hline \begin{center}$\mathfrak S_{2431}$\end{center}&\begin{center}\begin{ytableau}\color{Plum} 1\\ \color{blue} 2\color{Plum} 2&\color{blue} 2\\ \color{Plum} 3\\ \end{ytableau} + \begin{ytableau}\color{Plum} 1\\ \color{blue} 2\color{Plum} 2&\color{blue} 1\\ \color{Plum} 3\\ \end{ytableau} \end{center}&\begin{center}$x_1x_2^2x_3 + x_1^2x_2x_3$ \end{center} & \begin{center}$\mathfrak S_{4321}$\end{center}&\begin{center}\begin{ytableau} \resizebox{9pt}{!}{\color{blue} 1\color{Plum} 1\color{Rhodamine} 1}&\color{blue} 1\color{Plum} 1&\color{blue} 1\\ \color{Plum} 2\color{Rhodamine} 2&\color{Plum} 2\\ \color{Rhodamine} 3\\ \end{ytableau} \end{center}&\begin{center}$x_1^3x_2^2x_3$ \end{center}\\ 
\hline
\end{tabular}

\bigskip
 \caption{${\tt Prism}(w)$  and ${\mathfrak S}_w$ for all $w\in S_4$}
\end{table}

\subsection{Organization}
In Section~2 we present the general geometric perspective behind the rule and its proof.
In the case at hand, we need to study the Stanley-Reisner simplical complex associated to the Gr\"{o}bner limit of $X_w$; this
is done in Section~3. In Section~4, we collect some additional results and remarks.

\section{Main idea of the model and its proof}

Let ${\sf G}={\sf GL}_n$ and ${\sf B}$ and ${\sf B}^+$ the Borel subgroups of lower and upper triangular matrices in ${\sf G}$.  Identify the {\bf flag variety} with the coset space ${\sf B}\backslash {\sf G}$.  Let ${\sf T}$ be the maximal torus in ${\sf B}$.  Suppose ${\mathfrak X}\subset {\sf B}\backslash {\sf G}$ is an arbitrary subvariety and $\pi:{\sf G}\twoheadrightarrow {\sf B}\backslash {\sf G}$ is the natural projection.  Then 
\[X=\overline{\pi^{-1}({\mathfrak X})}\subseteq {\sf Mat}_{n\times n}\] 
carries a left ${\sf B}$ action and thus the action of ${\sf T}$. Therefore, one can speak of the equivariant cohomology class 
\[[X]_T\in H_{\sf T}({\sf Mat}_{n\times n})\cong {\mathbb Z}[x_1,\ldots,x_n].\] 
Moreover, the polynomial $[X]_{\sf T}$ is a 
coset representative under the Borel presentation of 
\[[{\mathfrak X}]\in H^{\star}({\sf B}\backslash {\sf G}, {\mathbb Z})\cong {\mathbb Z}[x_1,\ldots,x_n]/I^{S_n},\]
where $I^{S_n}$ is the ideal generated by (non-constant) elementary symmetric polynomials. This is a key perspective of work of A.~Knutson-E.~Miller \cite{Knutson.Miller:annals} when ${\mathfrak X}$ is a Schubert variety. 

Let $Y\subseteq {\sf Mat}_{n\times n}$ be an equidimensional, reduced union of coordinate subspaces.  Given ${\mathcal P}\subset n\times n$, we represent $\mathcal P$ visually as a collection of $+$'s in the $n\times n$ grid.  We say $\mathcal P$ is a {\bf plus diagram} for $Y$, if \[{\mathcal L}_{\mathcal P}:=\{M\in {\sf Mat}_{n\times n}: M_{i,j}=0 \text{\ if }(i,j)\in \mathcal P\}\subset Y.\]  

 Let ${\tt Plus}(Y)$ be the set of all such plus diagrams. Let ${\tt MinPlus}(Y)$ be the set of minimal plus diagrams, i.e., those ${\mathcal P}$ for which removing any $+$ 
would not return an element of ${\tt Plus}(Y)$. We refer to the union of plus diagrams as an {\bf overlay} to emphasize whenever $(i,j)$ is in $\mathcal P$ or $\mathcal P'$, the diagram for $\mathcal P\cup \mathcal P'$ also has a $+$ in position $(i,j)$.

Each ${\mathcal P}$ corresponds $1:1$ to a face of the Stanley-Reisner complex $\Delta_{Y}$. 
Let $\Delta_{n\times n}$ be the power set of $\{(i,j):1\leq i,j\leq n\}$.  Then $\Delta_Y\subseteq \Delta_{n\times n}$
and for each ${\mathcal P}$ one has 
the face \[{\mathcal F}_{\mathcal P}=\{(i,j):1\leq i,j\leq n \text{ and } (i,j)\not\in \mathcal P\}.\]  
The faces of $\Delta_Y$ are ordered by reverse containment of their
plus diagrams.
Thus, facets (maximal dimensional faces) of $\Delta_Y$ coincide with elements of ${\tt MinPlus}(Y)$.  
In addition,
taking the overlay of $\mathcal P\in {\tt Plus}(Y)$ and $\mathcal Q\in {\tt Plus}(Z)$ corresponds to intersecting faces in the Stanley-Reisner complex:  \[\mathcal F_{\mathcal P\cup \mathcal Q}=\mathcal F_{\mathcal P}\cap \mathcal F_{\mathcal Q}\in \Delta_{Y}\cap \Delta_{Z}.\]

Through the interpretation of $[Y]_T$ as a \emph{multidegree}, one
may express $[Y]_{T}$ as a generating series over ${\tt MinPlus}(Y)$.  
That is, 
\begin{equation}
\label{eqn:theweight}
[Y]_{T}=\sum_{{\mathcal P}\in {\tt MinPlus}(Y)} {\tt wt}({\mathcal P}), 
\text{ \ where ${\tt wt}({\mathcal P})=\prod_{i=1}^n x_i^{\text{$\#$ of $+$'s in row $i$}}$.}
\end{equation}
For details, the reader may
consult \cite{Miller.Sturmfels}; see Chapter~1 and Chapter~8 (and its notes). 

\begin{example}
Let $Y\subset {\sf Mat}_{2\times 2}$ be the zero locus of $z_{1,1}z_{1,2}$, i.e., the union of two coordinate
hyperplanes $\{z_{1,1}=0\}\cup \{z_{1,2}=0\}$. Then
\[\left(\begin{matrix}
+ & \cdot\\
\cdot & \cdot \end{matrix}\right),
\left(\begin{matrix}
\cdot & +\\
\cdot & \cdot \end{matrix}\right),
\left(\begin{matrix}
+ & +\\
\cdot & \cdot \end{matrix}\right)\in {\tt Plus}(Y)
\]
(the first two are in ${\tt MinPlus}(Y)$).
The complex $\Delta_Y$ is the $2$-dimensional ball depicted below. 
\[\begin{picture}(120,115)
\label{twoTriangles}
\put(10,20){\includegraphics[scale=0.6]{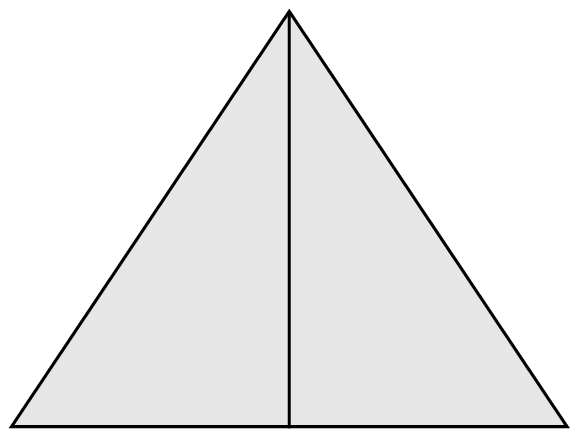}}
\put(29,37){\resizebox{30pt}{!}{$\left(\begin{matrix}+ & \cdot\\\cdot & \cdot \end{matrix}\right)$}}
\put(61,37){\resizebox{30pt}{!}{$\left(\begin{matrix}\cdot & +\\\cdot & \cdot \end{matrix}\right)$}}
\put(45,4){\resizebox{30pt}{!}{$\left(\begin{matrix}+ & +\\\cdot & + \end{matrix}\right)$}}
\put(-22,15){\resizebox{30pt}{!}{$\left(\begin{matrix}+ &\cdot \\+ & + \end{matrix}\right)$}}
\put(110,15){\resizebox{30pt}{!}{$\left(\begin{matrix}\cdot & +\\+ & + \end{matrix}\right)$}}
\put(45,105){\resizebox{30pt}{!}{$\left(\begin{matrix}+ & +\\+ & \cdot \end{matrix}\right)$}}
\end{picture}\]
Here $[Y]_T=2x_1$.\qed
\end{example}

Suppose $\prec$ is any term order on ${\mathbb C}[{\sf Mat}_{n\times n}]$ and $X':={\tt init}_{\prec}X$.
Since $X$ is ${\sf T}$-stable the same is true of 
$X'$; thus $[X']_{\sf T}$ is defined. Gr\"{o}bner degeneration preserves the ${\sf T}$-equivariant class, so 
$[X]_{\sf T}=[X']_{\sf T}$. Suppose $X'$ is reduced, and hence a reduced union of coordinate subspaces. 
Since ${\mathfrak X}$ was assumed to be irreducible, then $X$ is
irreducible. So by \cite[Theorem~1]{Kalkbrener.Sturmfels} the Stanley-Reisner complex $\Delta_{X'}$ of $X'$ is equidimensional. Hence we may apply
the discussion above using $Y=X'$ to compute $[X']_{\sf T}=[X]_{\sf T}$.

We are interested in understanding $\Delta_{X'}$ under certain hypotheses on
$X$. Assume that we have a collection of varieties 
$X,X_1,\ldots,X_m\subseteq V\cong \mathbb C^N$ such that
\begin{equation}
\label{eqn:spectro1}
X=X_1\cap X_2\cap \cdots\cap X_k.
\end{equation}
Assume $\prec$ is a term order on $\mathbb C[V]$ that defines a Gr\"obner degeneration of these varieties so that each Gr\"obner limit
\begin{equation}
\label{eqn:spectro2}
X':={\tt init}_\prec X, \: X_i':={\tt init}_{\prec} X_i \text{ \ (for $i=1,2,\ldots,k$) is reduced.}
\end{equation}
Finally, suppose
\begin{equation}
\label{eqn:spectro3}
X'=X_1'\cap X_2'\cap\cdots \cap X_k'.
\end{equation}
Call $\{X_i\}$ a {\bf $\prec$-spectrum} for $X$. 

To construct a cheap example, pick any Grobner basis $G=\{g_1,\ldots,g_M\}$ with square-free lead terms to define $X$. Now
partition $G=G_1\cup G_2\cup \cdots\cup G_s$ and set $X_k$ to be cut out
by $G_k$.  On the other hand, a motivating example is A.~Knutson \cite[Theorem~6]{Knutson:Frob}: given a term order $\prec$ (satisfying a hypothesis), there is a stratification of  
$V$ into a poset of varieties (ordered by inclusion) with the additional feature that each stratum $X$ admits a $\prec$-spectrum 
using higher strata. 

How can a $\prec$-spectrum be used to understand the combinatorics of $[X']_{\sf T}$? Here is a simple observation:

\excise{Given ${\mathcal P}\in {\tt Plus}(Y)$ and ${\mathcal Q}\in {\tt Plus}(Z)$, define the {\bf overlay} ${\mathcal P}\cup {\mathcal Q}$ to be the assignment of $+$'s to the $n\times n$ grid such that $+$ appears in position $(i,j)$ if a $+$ appears in that position in either ${\mathcal P}$ or ${\mathcal Q}$.  Clearly, the overlay of two plus diagrams corresponds to the intersection of faces of the corresponding Stanley-Reisner complexes, $\mathcal F_{\mathcal P\cup \mathcal Q}=\mathcal F_{\mathcal P}\cap \mathcal F_{\mathcal Q}$.} 

\begin{lemma}
\label{lemma:mainobs}
Let $\{X_i\}$ be a $\prec$-spectrum for $X$. Then 
\begin{itemize}
\item[(I)]
${\tt Plus}(X')=\{{\mathcal P}_1\cup\cdots\cup
{\mathcal P}_k: {\mathcal P}_i\in {\tt Plus}(X_i').\}$
\item[(II)] ${\tt MinPlus}(X')\subseteq \{{\mathcal P}_1\cup\cdots\cup
{\mathcal P}_k: {\mathcal P}_i\in {\tt MinPlus}(X_i').\}$
\end{itemize}
\end{lemma}

\begin{proof}
(I): Let $\mathcal P\in {\tt Plus}(X')$.  Then $\mathcal L_\mathcal P\subseteq X' \subseteq X_i'$ for all $i$.  Therefore $\mathcal P\in {\tt Plus}(X_i')$ and trivially $\mathcal P=\mathcal P\cup \ldots \cup \mathcal P$, proving ``$\subseteq$''. For
the other containment,  suppose ${\mathcal P}_i\in {\tt Plus}(X_i')$ for $1\leq i\leq k$ and let $\mathcal P=\mathcal P_1\cup \cdots \cup \mathcal P_k$.  Then 
$\mathcal L_{\mathcal P}=\mathcal L_{\mathcal P_1}\cap ...\cap \mathcal L_{\mathcal P_k}$ and hence $\mathcal L_{\mathcal P}\subseteq \mathcal L_{\mathcal P_i}\subseteq  X_i'$. So $\mathcal P \in {\tt Plus}(X_i')$  for each $i$, which implies $\mathcal P\in {\tt Plus}(X')$.

(II):  Let $\mathcal P\in {\tt MinPlus}(X')$. By (I), $\mathcal P\in {\tt Plus}(X_i')$ for each $i$.  Then there exists $\mathcal P_i\in  {\tt MinPlus}(X_i')$ so that $\mathcal P_i\subseteq \mathcal P$.  Then $\mathcal P\supseteq \mathcal P_1\cup \cdots \cup \mathcal P_k\in {\tt Plus}(X')$ by (I).  As $\mathcal P$ is minimal, this is an equality.
\end{proof}

Our point is that in good cases, the plus diagrams of $X_i'$ are ``simpler'' to
understand than those of $X$.  Lemma~\ref{lemma:mainobs}(II) says that one can think of
each ${\mathcal P}\in {\tt MinPlus}(X)$ as an overlay 
${\mathcal P}={\mathcal P}_{1}\cup\cdots\cup {\mathcal P}_k$ of these simpler ${\mathcal P}_i$.
Of course, this representation is not unique in general,
so one can make a \emph{choice} of representation for
each ${\mathcal P}$. The hope is to transfer understanding of the combinatorics of ${\tt MinPlus}(X_i)$ to the combinatorics of ${\tt MinPlus}(X)$. 

\section{Proof of the Theorem~\ref{theorem:main}}

We now carry out the ideas described in Section~2 in 
the case of Schubert varieties.

\subsection{Matrix Schubert varieties and Schubert polynomials}

The flag variety ${\sf B}\backslash {\sf G}$ decomposes into {\bf Schubert cells} 
${\mathfrak X}_w^\circ:={\sf B}\backslash {\sf B}w{\sf B}^+$ indexed by $w\in S_n$. The {\bf Schubert variety} 
is the Zariski-closure 
${\mathfrak X}_w:=\overline{\mathfrak X_w^\circ}$. 
The {\bf matrix Schubert variety} is 
\[X_w:=\overline{\pi^{-1}({\mathfrak X}_w)}\subset {\sf Mat}_{n\times n}.\]

Let $Z=(z_{ij})_{1\leq i,j\leq n}$ be the generic $n\times n$ matrix. The {\bf Schubert determinantal ideal} is
\[I_w=\langle\text{$r_w(i,j)+1$ minors of the
the northwest $i\times j$ submatrix of $Z$}\rangle\subset
\mathbb C[{\sf Mat}_{n\times n}].\] 
In \cite[Lemma~3.10]{Fulton:duke} it is proved that $I_w$ cuts out
$X_w$ scheme-theoretically. Moreover in \emph{loc. cit.} it is shown that $I_w$ is generated by the smaller set of generators coming from
those $(i,j)\in {\mathcal Ess}(w)$. 

By \cite[Theorem~A]{Knutson.Miller:annals}, 
\[[X_w]_{\sf T}={\mathfrak S}_w(x_1,\ldots,x_n)\in H_{\sf T}({\sf Mat}_{n\times n}).\]
Moreover, let $\prec_{\tt anti}$ be any {\bf antidiagonal term order} on 
$\mathbb C[{{\sf Mat}}_{n\times n}]$, i.e., one that picks off the antidiagonal term of any minor of $Z$. In \cite[Theorem~B]{Knutson.Miller:annals} it is shown that ${\tt MinPlus}(X_w')$ are in a transparent bijection with the $RC$-graphs  of \cite{Bergeron.Billey} (cf.~\cite{Fomin.Kirillov}). 

For each $e\in {\mathcal Ess}(w)$, there is a unique 
biGrassmannian permutation $u_e$ such that $r_{u_e}(e)=r_w(e)$ and  ${\mathcal Ess}(u_e)=\{e\}$ \cite{LS:trellis}. Let 
\[{\tt biGrass}(w):=\{u_e:e\in {\mathcal Ess}(w)\}=\{u_1,\ldots, u_k\}.\] 
Call $\{X_{u_1},\ldots, X_{u_k}\}$  the  {\bf biGrassmannian $\prec_{\tt anti}$-spectrum} for $X_w$. By \cite[Section~7.2]{Knutson:Frob},
$\{X_{u_i}\}$ indeed gives a $\prec_{\tt anti}$-spectrum for $X_w$ over ${\mathbb Q}$. This result can also be readily obtained (over ${\mathbb Z}$)
if one assumes the Gr\"{o}bner basis result \cite[Theorem B]{Knutson.Miller:annals}. (It should be emphasized that one of the points of 
 \cite[Section~7.2]{Knutson:Frob} is to reprove said Gr\"{o}bner basis theorem more easily.)

\begin{figure}
\includegraphics[scale=.6]{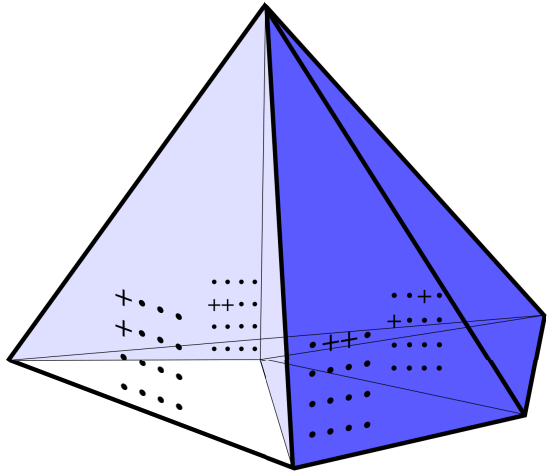}
\excise{
\begin{picture}(200,200)

\put(0,0){\includegraphics[scale=1.75]{plainSR.eps}}
\put(45,40){\resizebox{!}{17pt}{$
\left(\begin{matrix} 
\cdot & \cdot&\cdot&\cdot \\
+ & +&\cdot&\cdot \\
\cdot & \cdot&\cdot&\cdot \\
\cdot & \cdot&\cdot&\cdot \\
\end{matrix}\right)$}}
\put(147,60){\resizebox{!}{13pt}{$
\left(\begin{matrix} 
\cdot & \cdot&+&\cdot \\
+ & \cdot&\cdot&\cdot \\
\cdot & \cdot&\cdot&\cdot \\
\cdot & \cdot&\cdot&\cdot \\
\end{matrix}\right)$}}

\put(125,34){\resizebox{!}{19pt}{$
\left(\begin{matrix} 
+ & \cdot&\cdot&\cdot \\
+ & \cdot&\cdot&\cdot \\
\cdot & \cdot&\cdot&\cdot \\
\cdot & \cdot&\cdot&\cdot \\
\end{matrix}\right)$}}

\end{picture}
}

\caption{The Stanley-Reisner complexes for $X'_{1423}$ and $X'_{2314}$ intersect to give the complex for $X'_{2413}$. These complexes
are a multicone over the depicted complex.}
\label{srComplex}
\end{figure}

\begin{example}
$X=X_{2413}$ has biGrassmannian $\prec_{\tt anti}$-spectrum 
$\{X_1=X_{1423}, X_2=X_{2314}\}$.  
Here $\mathbb C[{\sf Mat}_{n\times n}]=\mathbb C[z_{i,j}: 1\leq i,j\leq 4]$ and one can
check:
\[I_{u_1}=\left\langle\left|\begin{matrix}
z_{1,1} & z_{1,2}\\ z_{2,1} & z_{2,2}\end{matrix}\right|,  \
\left|\begin{matrix}
z_{1,1} & z_{1,3}\\ z_{2,1} & z_{2,3}\end{matrix}\right|,  \
\left|\begin{matrix}
z_{1,2} & z_{1,3}\\ z_{2,2} & z_{2,3}\end{matrix}\right|\right\rangle, \quad
I_{u_2}=\left\langle z_{1,1}, z_{2,1}\right\rangle, \quad I_w=I_{u_1}+I_{u_2}.
\]
The $\prec_{\tt anti}$-Gr\"obner limits are defined by
\[I_{u_1}'=\langle z_{2,1}z_{1,2}, \ z_{2,1}z_{1,3}, z_{2,2}z_{1,3}\rangle, \quad
 I_{u_2}'=\langle z_{1,1}, z_{2,1}\rangle, \quad \ I_{w}'=I_{u_1}'+I_{u_2}'.\]

Since the prime decomposition of $I_{u_1}'$ is
\[I_{u_1'}=\langle z_{2,1}, z_{2,2}\rangle\cap \langle z_{2,1}, z_{1,3}\rangle\cap \langle z_{1,2}, z_{1,3}\rangle,\]
the facets of $\Delta_{X_1'}$ are labeled by: 
\begin{equation}
\label{eqn:firstfacets}
{\tt MinPlus}(X_1')=
\left\{
\left[\begin{array}{ccccc}
\cdot&\cdot&\cdot&\cdot\\
+&+&\cdot&\cdot\\
\cdot&\cdot&\cdot&\cdot\\
\cdot&\cdot&\cdot&\cdot\\
\end{array}\right],
\left[\begin{array}{ccccc}
\cdot&\cdot&+&\cdot\\
+&\cdot&\cdot&\cdot\\
\cdot&\cdot&\cdot&\cdot\\
\cdot&\cdot&\cdot&\cdot\\
\end{array}\right],
\left[\begin{array}{ccccc}
\cdot&+&+&\cdot\\
\cdot&\cdot&\cdot&\cdot\\
\cdot&\cdot&\cdot&\cdot\\
\cdot&\cdot&\cdot&\cdot\\
\end{array}\right]
\right\}.
\end{equation}
In Figure~\ref{srComplex}, these correspond to the indicated tetrahedra. 

Similarly, there is a single facet for $X'_{2}$ associated to the prime ideal $I_{u_2}$, labeled by:
\begin{equation}
\label{eqn:secondfacet}
\left\{\left[\begin{array}{ccccc}
+&\cdot&\cdot&\cdot\\
+&\cdot&\cdot&\cdot\\
\cdot&\cdot&\cdot&\cdot\\
\cdot&\cdot&\cdot&\cdot\\
\end{array}\right]
\right\}.
\end{equation}
This facet corresponds to the remaining tetrahedron. 

There are precisely two minimal overlays of the plus diagrams of (\ref{eqn:firstfacets}) with the plus diagram of (\ref{eqn:secondfacet}):
\[\left\{
\left[\begin{array}{ccccc}
+&\cdot&\cdot&\cdot\\
+&+&\cdot&\cdot\\
\cdot&\cdot&\cdot&\cdot\\
\cdot&\cdot&\cdot&\cdot\\
\end{array}\right],
\left[\begin{array}{ccccc}
+&\cdot&+&\cdot\\
+&\cdot&\cdot&\cdot\\
\cdot&\cdot&\cdot&\cdot\\
\cdot&\cdot&\cdot&\cdot\\
\end{array}\right]
\right\}.\]
This agrees with the prime decomposition
$I_w'=\langle z_{1,1}, z_{1,3}, z_{2,1}\rangle\cap \langle z_{1,1}, z_{2,1}, z_{2,2}\rangle$. Geometrically, these label the facets of $\Delta_{X'}$, pictured as light blue triangles in Figure~\ref{srComplex}.

Finally, applying the discussion of Section~2 (cf.~(\ref{eqn:theweight})) we see that
\[{\mathfrak S}_{u_1}=x_2^2+x_1 x_2+x_1^2,  \ \ 
{\mathfrak S}_{u_2}=x_1 x_2, \text{\  \ and 
${\mathfrak S}_w=x_1 x_2^2+x_1^2 x_2$}\]
(where the terms in each Schubert polynomial correspond respectively to the plus diagrams listed above).
\qed
\end{example}

\begin{example}[Digression: diagonal term orders]
Fix $w=2143\in S_4$ and let $\prec_{\tt diag}$ be any {\bf diagonal term order} 
on ${\mathbb C}[{\sf Mat}_{n\times n}]$,
i.e., any order that picks the diagonal term of a minor as the lead term. One has that
${X}_{2143}={X}_{2134}\cap {X}_{1243}$ (reduced intersection). Now
\[{I}_{2134}'=\langle z_{11}\rangle \mbox{\ and \ }
I_{1243}'=\langle z_{11}z_{22}z_{33}\rangle.\]
However, 
\[
{I}_{2143}'=\langle z_{11}, z_{12}z_{21}z_{33} \rangle\neq
{I}_{2134}'+ {I}_{1243}'=\langle z_{11}, z_{11}z_{22}z_{33}\rangle=I'_{2134}.\]
So $\{X_{2134}, X_{1243}\}$ is not a $\prec_{\tt diag}$-spectrum for $X_w$.\qed   
\end{example}

A permutation is {\bf vexillary} if it is $2143$-avoiding; see \cite[Section~2.2.1]{Manivel} for details. The following is not needed in the proof of Theorem~\ref{theorem:main}:

\begin{proposition}
$\{X_{u_1},\ldots, X_{u_k}\}$ is
a $\prec_{\tt diag}$-spectrum for $X_w$ if and only if $w$ is vexillary.
\end{proposition}
\begin{proof}
Assume $w$ is vexillary. Then by \cite[Section~1.4]{KMY}, the essential
minors define a $\prec_{\tt diag}$-Gr\"{o}bner basis for $I_w$. 
The same is true of $I_{u_i}$ since $u_i$ is biGrassmannian and therefore
also vexillary. Since the (Gr\"{o}bner) essential minors of $I_w$ are the concatentation
of the (Gr\"{o}bner) essential minors of the $I_{u_i}$'s, the spectrum claim
follows.

For the converse, assume $w$ is not vexillary, but $\{X_{u_1},\ldots, X_{u_k}\}$
is a $\prec_{\tt diag}$-spectrum for $X_w$. Again, we know the essential minors
of $I_{u_i}$ form a $\prec_{\tt diag}$-Gr\"{o}bner basis. By the spectrum assumption, the concatenation of these $k$-many Gr\"{o}bner basis is a $\prec_{\tt diag}$-Gr\"{o}bner basis for $I_w$. However this concatenated Gr\"{o}bner basis is
the set of essential generators for $I_w$. This directly contradicts \cite[Theorem~6.1]{KMY}.
\end{proof}

\subsection{Multi-plus diagrams}
The technical core of our proof is to analyze the combinatorics of overlays of plus diagrams for the biGrassmannian $\prec_{\tt anti}$-spectrum $\{X_{u_1},\ldots,X_{u_k}\}$. Let 
\[{\tt Multi}(w)=\prod_{i=1}^k{\tt MinPlus}(X'_{u_i})\] 
be the set of  {\bf multi-plus diagrams} for $w$: we represent 
$(\mathcal P_1,\ldots,\mathcal P_k)\in {\tt Multi}(w)$ as a placement of 
colored $+$'s in a single $n\times n$ grid, where 
$(a,b)$ has a + of color $u_i$ if $(a,b)\in {\mathcal P}_i$.  

By Lemma~\ref{lemma:mainobs}(I), there is a map \[{\tt supp}:{\tt Multi}(w)\rightarrow {\tt Plus}(X'_w)\] given by $(\mathcal P_1,\ldots, \mathcal P_k)\mapsto \mathcal P_1\cup \ldots \cup \mathcal P_k$.  Call $\mathcal P_1\cup \ldots \cup \mathcal P_k$ the {\bf support} of $(\mathcal P_1,\ldots,\mathcal P_k)$. Central to our study is
\[{\tt Multi}(\mathcal P):={\tt supp}^{-1}(\mathcal P).\]  

\begin{example}
Let $w=42513$. Then ${\tt biGrass}(w)=\{{\color{blue}41235},{\color{Plum} 23415},{\color{Rhodamine}14523}\}$. Now,
\[\mathcal P=\left[\begin{array}{ccccc}
+&+&+&\cdot&\cdot\\
+&\cdot&+&\cdot&\cdot\\
+&\cdot&\cdot&\cdot&\cdot\\
\cdot&\cdot&\cdot&\cdot&\cdot\\
\cdot&\cdot&\cdot&\cdot&\cdot\\
\end{array}\right]\in {\tt MinPlus}(X_w').\]
One can check that
\[{\tt Multi}({\mathcal P})=\left\{\left[\begin{array}{ccccc}
{\color{Plum}+}{\color{blue}+}&{\color{blue}+}{\color{Rhodamine}+}&{\color{blue}+}{\color{Rhodamine}+}&\cdot&\cdot\\
{\color{Plum}+}&\cdot&{\color{Rhodamine}+}&\cdot&\cdot\\
{\color{Plum}+}{\color{Rhodamine}+}&\cdot&\cdot&\cdot&\cdot\\
\cdot&\cdot&\cdot&\cdot&\cdot\\
\cdot&\cdot&\cdot&\cdot&\cdot\\
\end{array}\right], \qquad \left[\begin{array}{ccccc}
{\color{Plum}+}{\color{blue}+}&{\color{blue}+}&{\color{blue}+}{\color{Rhodamine}+}&\cdot&\cdot\\
{\color{Plum}+}{\color{Rhodamine}+}&\cdot&{\color{Rhodamine}+}&\cdot&\cdot\\
{\color{Plum}+}{\color{Rhodamine}+}&\cdot&\cdot&\cdot&\cdot\\
\cdot&\cdot&\cdot&\cdot&\cdot\\
\cdot&\cdot&\cdot&\cdot&\cdot\\
\end{array}\right]\right \}.\]
\qed
\end{example}

\subsection{Local moves on plus diagrams}

A {\bf southwest move} is the following local 
operation on a plus diagram:
\begin{equation}
\label{local1}
\left[\begin{array}{cc}\cdot &+\\\cdot &\cdot \end{array}\right]\mapsto\left[\begin{array}{cc}\cdot &\cdot\\ +& \cdot \end{array}\right].
\end{equation}
The inverse operation is a {\bf northeast move}:
\begin{equation}
\label{local2}
\left[\begin{array}{cc}\cdot &\cdot\\+ &\cdot \end{array}\right]\mapsto\left[\begin{array}{cc}\cdot &+\\ \cdot& \cdot \end{array}\right].
\end{equation}

Suppose ${\mathcal Ess}(u)=\{(i,j)\}$. Define $\mathcal D_{\tt bot}(u)\in {\tt MinPlus}(X_u')$ as the $(i-r_u(i,j))\times (j-r_u(i,j))$ rectangle of $+$'s, with southwest corner in row $i$ and column $1$.
The following is well-known, and is a consequence (by specialization)
of the chute and ladder moves of \cite[Theorem~3.7]{Bergeron.Billey}:
\begin{lemma}
\label{lemma:localmoves}
Let $u\in S_n$ be biGrassmannian.
\begin{itemize}
\item[(I)] ${\tt MinPlus}(X_u')$ is connected and closed under the moves (\ref{local1}) and (\ref{local2}).
\item[(II)] Each ${\mathcal P}\in {\tt MinPlus}(X_u')$ can be
obtained from $\mathcal D_{\tt bot}(u)$ using only
the moves (\ref{local2}).
\end{itemize}
\end{lemma}

Define a partial order on ${\tt MinPlus}(X_u')$ by taking
the transitive closure of the covering relation $\mathcal P<\mathcal P'$ if $\mathcal P'$ is obtained from $\mathcal P$ by a northeast local move (\ref{local2}). Let $<'$ be the
partial order on ${\tt Multi}(w)$ defined as the 
${\mathcal Ess}(w)$-factor Cartesian product of $<$. That is
$({\mathcal P}_1,\ldots,{\mathcal P}_k)<'({\mathcal Q}_1,\ldots,{\mathcal Q}_k)$
if and only if ${\mathcal P}_i<{\mathcal Q}_i$ for each $i$.
Then $<'$ induces a partial order on
${\tt Multi}(\mathcal P)\subseteq {\tt Multi}(w)$.

Given $(\mathcal P_1,\ldots,\mathcal P_m)\in {\tt Multi}(w)$, a {\bf long move} is a repeated application of (\ref{local1})
(respectively, (\ref{local2})) to a single $+$ appearing in one of the ${\mathcal P}_i$'s. 
\excise{
a composition of moves (3.3) or (3.4), moving a + in $\mathcal P_i$ within its diagonal to give $\mathcal P_i'$ so that ${\tt supp}(\mathcal P_1,\ldots,\mathcal P_m)={\tt supp}(\mathcal P_1,\ldots\mathcal P_i',\ldots,\mathcal P_m)$.}
Recall that a {\bf lattice} is a partially ordered set in which every two elements $x$ and $y$ have a least upper bound (join) 
$x\vee y$ (join)
and a unique greatest lower bound $x\wedge y$ (meet). 	It is basic that a Cartesian product of lattices is a lattice.

\begin{theorem}
\label{theorem:overlay}
Let $w\in S_n$ and $\mathcal P\in {\tt MinPlus}(X'_w)$.  
\begin{itemize}
\item[(I)]  ${\tt Multi}(\mathcal P)$ is connected by long moves.
\item[(II)] Each $({\tt MinPlus}(X_{u_i}'),<)$ is a lattice. Consequently, $({\tt Multi}(w),<')$
is a lattice.
\item[(III)] $({\tt Multi}(\mathcal P),<')$ is a sublattice of $({\tt Multi}(w),<')$. 
\end{itemize}
\end{theorem}

\begin{example}
\label{exa:5361724}  
Let $w=5361724$. Fix \[\mathcal P=
\left[\begin{array}{ccccccc}
+		&+			&+&+&\cdot	&\cdot	&\cdot	\\
+		& + 			&\cdot		&+	&\cdot	&\cdot	&\cdot	\\
+		& + 			&\cdot		&+	&\cdot	&\cdot	&\cdot	\\
\cdot		&+ 			&\cdot		&\cdot		&\cdot	&\cdot	&\cdot	\\
\cdot		&\cdot		&\cdot		&\cdot		&\cdot	&\cdot	&\cdot	\\
\cdot		&\cdot		&\cdot		&\cdot		&\cdot	&\cdot	&\cdot	\\
\cdot		&\cdot 		&\cdot		&\cdot		&\cdot	&\cdot	&\cdot	\\
\end{array}\right]\in {\tt MinPlus}(X_w').\]
Here ${\tt biGrass}(w)=\{{\color{Rhodamine}5123467}, {\color{red}3451267}, {\color{blue}1562347},
      {\color{ForestGreen}1345627}, {\color{orange}1256734}\}$.  Figure~\ref{fig:poset} shows the Hasse diagram for ${\tt Multi}(\mathcal P)$. The poset is a lattice, agreeing with Theorem~\ref{theorem:overlay}(III).\qed 
\end{example}

\begin{figure}[t]
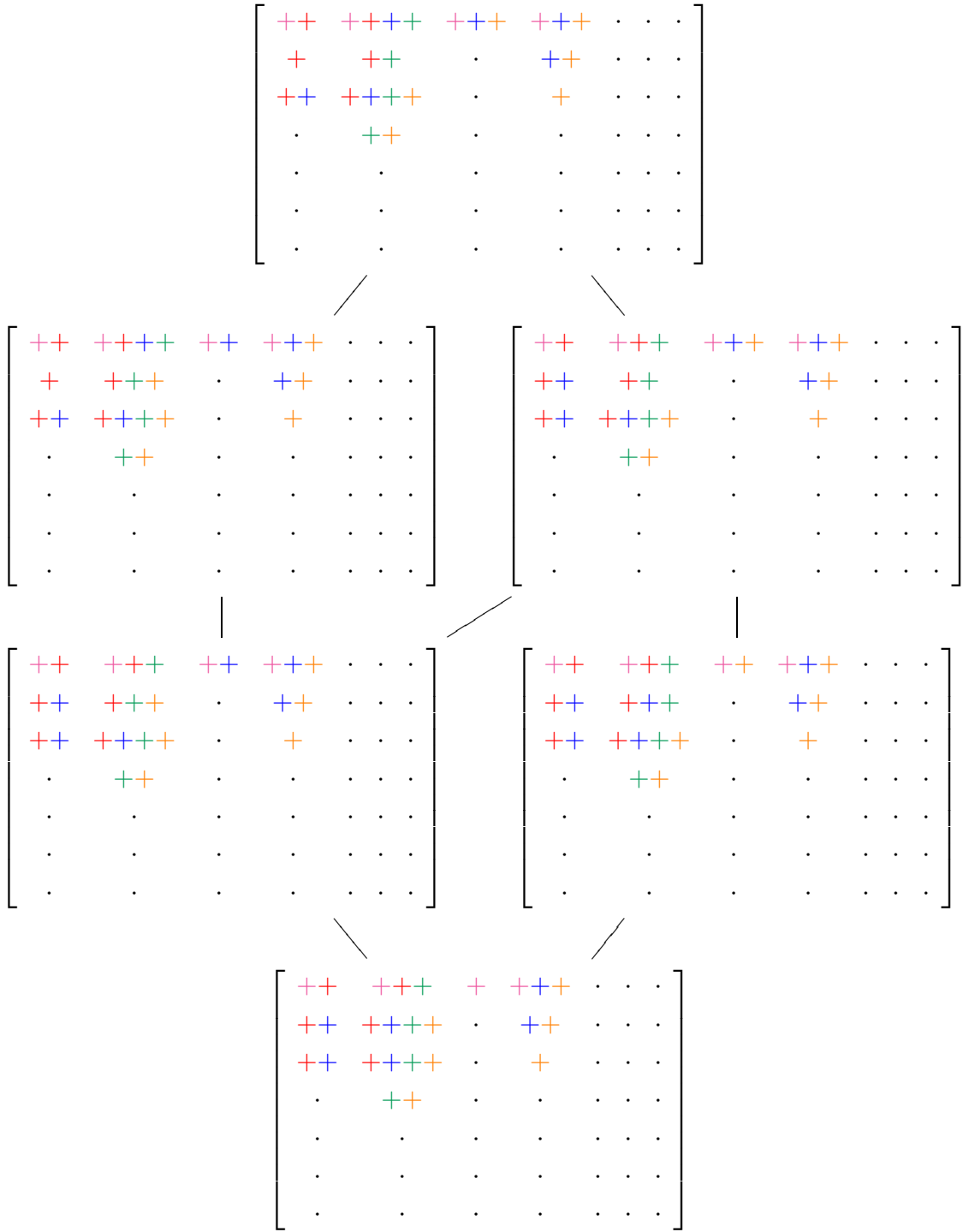

\begin{center}
$\xygraph{
!{<0cm,0cm>;<1cm,0cm>:<0cm,1cm>::}
!{(4,15) }*+{\left[\begin{array}{ccccccc}
\plusr \plusg 	&\plusr \plusg \plusb\plusy	&\plusr \plusb \pluso&\plusr \plusb \pluso&\cdot	&\cdot	&\cdot	\\
\plusg		&\plusg \plusy 			&\cdot		&\plusb \pluso	&\cdot	&\cdot	&\cdot	\\
\plusg \plusb	&\plusg \plusb \plusy \pluso	&\cdot		&\pluso		&\cdot	&\cdot	&\cdot	\\
\cdot		&\plusy \pluso			&\cdot		&\cdot		&\cdot	&\cdot	&\cdot	\\
\cdot		&\cdot			&\cdot		&\cdot		&\cdot	&\cdot	&\cdot	\\
\cdot		&\cdot			&\cdot		&\cdot		&\cdot	&\cdot	&\cdot	\\
\cdot		&\cdot 			&\cdot		&\cdot		&\cdot	&\cdot	&\cdot	\\
\end{array}\right]}="a"
!{(0,10) }*+{\left[\begin{array}{ccccccc}
\plusr \plusg 	&\plusr \plusg \plusb\plusy	&\plusr \plusb &\plusr \plusb \pluso&\cdot	&\cdot	&\cdot	\\
\plusg		&\plusg \plusy \pluso		&\cdot		&\plusb \pluso	&\cdot	&\cdot	&\cdot	\\
\plusg \plusb	&\plusg \plusb \plusy \pluso	&\cdot		&\pluso		&\cdot	&\cdot	&\cdot	\\
\cdot		&\plusy \pluso			&\cdot		&\cdot		&\cdot	&\cdot	&\cdot	\\
\cdot		&\cdot			&\cdot		&\cdot		&\cdot	&\cdot	&\cdot	\\
\cdot		&\cdot			&\cdot		&\cdot		&\cdot	&\cdot	&\cdot	\\
\cdot		&\cdot 			&\cdot		&\cdot		&\cdot	&\cdot	&\cdot	\\
\end{array}\right]}="b"
!{(8,10) }*+{\left[\begin{array}{ccccccc}
\plusr \plusg 	&\plusr \plusg \plusy		&\plusr \plusb \pluso&\plusr \plusb \pluso&\cdot	&\cdot	&\cdot	\\
\plusg	\plusb	&\plusg \plusy 			&\cdot		&\plusb \pluso	&\cdot	&\cdot	&\cdot	\\
\plusg \plusb	&\plusg \plusb \plusy \pluso	&\cdot		&\pluso		&\cdot	&\cdot	&\cdot	\\
\cdot		&\plusy \pluso			&\cdot		&\cdot		&\cdot	&\cdot	&\cdot	\\
\cdot		&\cdot			&\cdot		&\cdot		&\cdot	&\cdot	&\cdot	\\
\cdot		&\cdot			&\cdot		&\cdot		&\cdot	&\cdot	&\cdot	\\
\cdot		&\cdot 			&\cdot		&\cdot		&\cdot	&\cdot	&\cdot	\\
\end{array}\right]}="c"
!{(0,5) }*+{\left[\begin{array}{ccccccc}
\plusr \plusg 	&\plusr \plusg \plusy		&\plusr \plusb &\plusr \plusb \pluso&\cdot	&\cdot	&\cdot	\\
\plusg	\plusb	&\plusg \plusy \pluso		&\cdot		&\plusb \pluso	&\cdot	&\cdot	&\cdot	\\
\plusg \plusb	&\plusg \plusb \plusy \pluso	&\cdot		&\pluso		&\cdot	&\cdot	&\cdot	\\
\cdot		&\plusy \pluso			&\cdot		&\cdot		&\cdot	&\cdot	&\cdot	\\
\cdot		&\cdot			&\cdot		&\cdot		&\cdot	&\cdot	&\cdot	\\
\cdot		&\cdot			&\cdot		&\cdot		&\cdot	&\cdot	&\cdot	\\
\cdot		&\cdot 			&\cdot		&\cdot		&\cdot	&\cdot	&\cdot	\\
\end{array}\right]}="d"
!{(8,5) }*+{\left[\begin{array}{ccccccc}
\plusr \plusg 	&\plusr \plusg \plusy		&\plusr  \pluso&\plusr \plusb \pluso&\cdot	&\cdot	&\cdot	\\
\plusg	\plusb	&\plusg \plusb \plusy 		&\cdot		&\plusb \pluso	&\cdot	&\cdot	&\cdot	\\
\plusg \plusb	&\plusg \plusb \plusy \pluso	&\cdot		&\pluso		&\cdot	&\cdot	&\cdot	\\
\cdot		&\plusy \pluso		&\cdot		&\cdot		&\cdot	&\cdot	&\cdot	\\
\cdot		&\cdot			&\cdot		&\cdot		&\cdot	&\cdot	&\cdot	\\
\cdot		&\cdot			&\cdot		&\cdot		&\cdot	&\cdot	&\cdot	\\
\cdot		&\cdot 			&\cdot		&\cdot		&\cdot	&\cdot	&\cdot	\\
\end{array}\right]}="e"
!{(4,0) }*+{\left[\begin{array}{ccccccc}
\plusr \plusg 	&\plusr \plusg \plusy		&\plusr  		&\plusr \plusb \pluso&\cdot	&\cdot	&\cdot	\\
\plusg	\plusb	&\plusg \plusb \plusy \pluso	&\cdot		&\plusb \pluso	&\cdot	&\cdot	&\cdot	\\
\plusg \plusb	&\plusg \plusb \plusy \pluso	&\cdot		&\pluso		&\cdot	&\cdot	&\cdot	\\
\cdot		&\plusy \pluso		&\cdot		&\cdot		&\cdot	&\cdot	&\cdot	\\
\cdot		&\cdot			&\cdot		&\cdot		&\cdot	&\cdot	&\cdot	\\
\cdot		&\cdot			&\cdot		&\cdot		&\cdot	&\cdot	&\cdot	\\
\cdot		&\cdot 			&\cdot		&\cdot		&\cdot	&\cdot	&\cdot	\\
\end{array}\right]}="f"
"a"-"b"  "a"-"c" "b"-"d" "c"-"e" "d"-"f" "e"-"f" "d"-"c"}
$
\end{center}
\caption{The subposet of ${\tt Multi}(5361724)$ with support $\mathcal P$ (cf.~Example~\ref{exa:5361724}).}
\label{fig:poset}
\end{figure}

We use the following ordering of the $+$'s of $\mathcal D_{\tt bot}(u)$.
\begin{equation}
\label{eqn:ordering}
\left[\begin{array}{cccc}
+_7&\rdots&&\\
+_4&+_8&\rdots&\\
+_2&+_5&+_9&\rdots\\
+_1&+_3&+_6&+_{10}\\
\end{array}\right].
\end{equation}
In words, order the $+$'s along diagonals, from northwest to southeast,
where $+_1$ is at the southwest corner of ${\mathcal D}_{\tt bot}$.

In view of Lemma~\ref{lemma:localmoves}(II), there
is a bijection between the $+$'s of
$\mathcal D_{\tt bot}(u)$ and any ${\mathcal P}\in
{\tt MinPlus}(X_u')$. Hence the ordering (\ref{eqn:ordering})
induces an ordering $+_1,+_2,\ldots$ of the $+$'s of ${\mathcal P}$.

The following two lemmas hold by Lemma~\ref{lemma:localmoves} and an easy
induction.

\begin{lemma}
\label{lemma:weaklysouthwest}
If in ${\mathcal D}_{\tt bot}(u)\in {\tt MinPlus}(X_{u}')$ the label $+_a$ is weakly southwest of $+_b$, then the same is true for all
${\mathcal P}\in {\tt MinPlus}(X_u')$.
\end{lemma}

For an antidiagonal $D$, 
let $D_{\rm left}$ be the antidiagonal adjacent to $D$ and to its left. Similarly let $D_{\rm right}$ be the antidiagonal adjacent to $D$ and to its right. 
\begin{lemma}
\label{lemma:relativeorder}
Fix $\mathcal P\in {\tt MinPlus}(X_u')$. Fix an antidiagonal $D$. Let $+_a\in D$ and suppose $+_b$ is in $D$, $D_{\rm left}$ or $D_{\rm right}$ so that $+_b$ is weakly southwest of $+_a$.  Then for any $\mathcal P'\in {\tt MinPlus}(X_u')$, $+_b$ is weakly southwest of $+_a$.
\end{lemma}

\begin{proposition}

\label{prop:moveorder}

\begin{itemize}
\item[(I)] Let $u$ be a biGrassmannian permutation and $\mathcal P, \mathcal P'\in {\tt MinPlus}(X'_u)$. 
Consider the following lists of indices:
\[{\tt SAME}=(a: +_a \text{\ appears in the same location in ${\mathcal P}$ and ${\mathcal P}'$}),\]
\[{\tt SW}=(a: +_a \text{\ in ${\mathcal P}'$ is strictly southwest of $+_a$ in ${\mathcal P}$}),\]
and
\[{\tt NE}=(a: +_a \text{\ in ${\mathcal P}'$ is strictly northeast of $+_a$ in ${\mathcal P}$}).\]
Let $\Lambda$ be the sequence contained by the concatenation
${\tt SAME}, {\tt SW}, {\tt NE}$ where {\tt SAME} and {\tt SW} are listed 
in increasing order whereas the
elements of {\tt NE} are listed in decreasing order.

Then there exists a sequence:
\begin{equation}
\label{eqn:whatever123}
{\mathcal P}:={\mathcal P}_{1}\mapsto {\mathcal P}_2\mapsto {\mathcal P}_3 \mapsto \cdots
\mapsto {\mathcal P}_h\mapsto {\mathcal P}_{h+1}\mapsto\cdots
\mapsto {\mathcal P}_{\ell}\mapsto
{\mathcal P}_{\ell+1}:={\mathcal P}'
\end{equation}
where 
\begin{itemize}
\item[(i)] $\ell=\ell(u)=|\lambda(u)|$
\item[(ii)] each ${\mathcal P}_i\in {\tt MinPlus}(X_u')$;
\item[(iii)] ${\mathcal P}_h\mapsto {\mathcal P}_{h+1}$ is an application of a long move to $+_{\Lambda_h}$.
\end{itemize}
\item[(II)] $({\tt MinPlus}(X'_u),<)$ is a lattice.
\end{itemize}
\end{proposition}

\begin{example}
Let $u=1267345$.  The ordering (\ref{eqn:ordering}) is as follows (where for brevity we ignore the unused bottom three rows of the $7\times 7$ ambient square):
\[\mathcal D_{\tt bot}(u)\!=\!\left [ \begin{array}{ccccccc}
\cdot &\cdot& \cdot& \cdot &\cdot& \cdot & \cdot\\ 
\cdot & \cdot & \cdot & \cdot &\cdot& \cdot & \cdot\\ 
+_2& +_4& +_6 & \cdot &\cdot& \cdot & \cdot\\ 
+_1 & +_3 & +_5 & \cdot &\cdot& \cdot & \cdot\\ 
\end{array}\right].\]
\[\text{  Let 
$
\mathcal P\!=\!\left [ \begin{array}{ccccccc}
\cdot &\cdot& \cdot& \cdot &\cdot& \cdot & \cdot\\ 
\cdot & +_2 & +_4 & +_6 &\cdot& \cdot & \cdot\\ 
\cdot& \cdot& \cdot & +_5 &\cdot& \cdot & \cdot\\ 
+_1 & +_3 & \cdot & \cdot &\cdot& \cdot & \cdot\\ 
\end{array}\right]
$
 and \ 
$\mathcal P'\!=\!\left [ \begin{array}{ccccccc}
\cdot &\cdot& \cdot& \cdot &+_6& \cdot & \cdot\\ 
\cdot & \cdot & +_4 & \cdot &+_5& \cdot & \cdot\\ 
+_2& \cdot& +_3 & \cdot &\cdot& \cdot & \cdot\\ 
+_1 & \cdot & \cdot & \cdot &\cdot& \cdot & \cdot\\ 
\end{array}\right]
$.}\]
Here ${\tt SAME}=(1,4), {\tt SW}=(2), \text{ and } {\tt NE}=(6,5,3)$.  
Therefore $\Lambda=(1,4,2,6,5,3)$.

Removing trivial long moves for ${\tt SAME}$, the sequence (\ref{eqn:whatever123}) consists of the following moves (where we have underlined $+_{\Lambda_h}$
for emphasis):
 \[\mathcal P_3\!=\!\left [ \begin{array}{ccccccc}
\cdot &\cdot& \cdot& \cdot &\cdot& \cdot & \cdot\\ 
\cdot & \underline{+_2} & +_4 & +_6 &\cdot& \cdot & \cdot\\ 
\cdot& \cdot& \cdot & +_5 &\cdot& \cdot & \cdot\\ 
+_1 & +_3 & \cdot & \cdot &\cdot& \cdot & \cdot\\ 
\end{array}\right]
\mapsto
{\mathcal P}_4=\!\left [ \begin{array}{ccccccc}
\cdot &\cdot& \cdot& \cdot &\cdot& \cdot & \cdot\\ 
\cdot & \cdot & +_4 & \underline{+_6} &\cdot& \cdot & \cdot\\ 
+_2& \cdot& \cdot & +_5 &\cdot& \cdot & \cdot\\ 
+_1 & +_3 & \cdot & \cdot &\cdot& \cdot & \cdot\\ 
\end{array}\right]\]
\[
\mapsto
{\mathcal P}_5=\!\left [ \begin{array}{ccccccc}
\cdot &\cdot& \cdot& \cdot &+_6& \cdot & \cdot\\ 
\cdot & \cdot & +_4 & \cdot &\cdot& \cdot & \cdot\\ 
+_2& \cdot& \cdot & \underline{+_5} &\cdot& \cdot & \cdot\\ 
+_1 & +_3 & \cdot & \cdot &\cdot& \cdot & \cdot\\ 
\end{array}\right]
\mapsto 
\mathcal P_6=\!\left [ \begin{array}{ccccccc}
\cdot &\cdot& \cdot& \cdot &+_6& \cdot & \cdot\\ 
\cdot & \cdot & +_4 & \cdot &+_5& \cdot & \cdot\\ 
+_2& \cdot& \cdot & \cdot &\cdot& \cdot & \cdot\\ 
+_1 & \underline{+_3} & \cdot & \cdot &\cdot& \cdot & \cdot\\ 
\end{array}\right]\]
\[
\mapsto 
\mathcal P_7=\!\left [ \begin{array}{ccccccc}
\cdot &\cdot& \cdot& \cdot &+_6& \cdot & \cdot\\ 
\cdot & \cdot & +_4 & \cdot &+_5& \cdot & \cdot\\ 
+_2& \cdot& +_3  & \cdot &\cdot& \cdot & \cdot\\ 
+_1 & \cdot& \cdot & \cdot &\cdot& \cdot & \cdot\\ 
\end{array}\right]
\]
In particular, ${\mathcal P}_4={\mathcal P}_{\#{\tt SAME}+\#{\tt SW}+1}$
is $\mathcal P\wedge\mathcal P'$.\qed
\end{example}

\noindent
\emph{Proof of Proposition~\ref{prop:moveorder}:}
(I): We show that for each $1\leq h\leq \ell$ 
one can give the desired long move. Clearly, we can use the trivial long move for $1\leq h \leq \#{\tt SAME}$. 

For $\#{\tt SAME}< h\leq \#{\tt SAME}+\#{\tt SW}$ we have ${\Lambda_h}\in {\tt SW}$. 
Suppose the $+_{\Lambda_h}$ in ${\mathcal P}'$ is in row $r'$ and suppose the $+_{\Lambda_h}$ in ${\mathcal P}$ is in row $r$ (where we have
assumed $r'>r$ in matrix notation). Let $D$ be the antidiagonal that $+_{\Lambda_h}$ sits in (in either ${\mathcal P}$ and ${\mathcal P}'$).

\begin{claim}
\label{claim:soso234}
In ${\mathcal P}_{h}$, there is no $+_b$ in rows
\begin{enumerate}
\item $r,r+1,\ldots, r'-1$ of $D_{\rm left}$; 
\item $r+1,r+2,\ldots, r'$ of $D_{\rm right}$; or
\item $r+1,r+2,\ldots, r'$ of $D$.
\end{enumerate}
\end{claim}
\noindent
\emph{Proof of Claim~\ref{claim:soso234}:}
If $\Lambda_h=b$ then by definition the unique $+_{\Lambda_h}=+_b$ is in row
$r$ of $D$, and in particular not in (1), (2) or (3).
If $\Lambda_h<b$, by the definition of the ordering
on $+$'s combined with Lemma~\ref{lemma:relativeorder}, $+_b$ is not weakly southwest of $+_{\Lambda_h}$.  Thus we may assume $\Lambda_h>b$.  Since $\Lambda_h>b$ the position of $+_b$ in $\mathcal P'$ will either be the same (if $b\in {\tt SAME}$ or $b\in {\tt SW}$), or strictly northeast of its position in $\mathcal P_h$ (if $b\in {\tt NE}$).  The  position of $+_{\Lambda_h}$ in $\mathcal P'$ is row $r'$ of $D$, which is weakly southwest of the position of $+_b$ in $\mathcal P'$. However, by the assumption that $+_b$ is in (1), (2) or (3), we see that
in $\mathcal P_h$, $+_b$ is weakly southwest of $+_{\Lambda_h}$. Hence we
obtain a contradiction of Lemma~\ref{lemma:relativeorder}. \qed


In view of Claim~\ref{claim:soso234} we may apply the long move ${\mathcal P}_{h}\mapsto {\mathcal P}_{h+1}$ that moves
the $+_{\Lambda_{h}}$ in row $r$ of ${\mathcal P}_{h}$ to row $r'$, showing (iii). Since each long move is by definition a composition of southwest moves,
(ii) holds by Lemma~\ref{lemma:localmoves}(I).

Finally, for $\Lambda_h\in {\tt NE}$ we have a long move for the same reasons
(\emph{mutatis mutandis}) as in our analysis above of $\Lambda_h\in {\tt SW}$.
(Alternatively, let ${\mathcal W}'={\mathcal P}_{\#{\tt SAME}+\#{\tt SW}+1}$
and ${\mathcal W}={\mathcal P}'$. Then by the above arguments there are
southwest long moves connecting ${\mathcal W}$ to ${\mathcal W}'$. Then 
we can reverse these moves to give the desired northeast long moves from
${\mathcal W}'$ to ${\mathcal P}'$.)

Since the list {\tt SAME}, {\tt SW}, {\tt NE} is of length $\ell$, (i)
holds trivially.

(II): We will only construct ${\mathcal P}\wedge {\mathcal P}'\in {\tt MinPlus}(X_u')$ (the construction of 
$\mathcal P \vee \mathcal P'$ is entirely analogous). 
From (I) we have the sequence of long moves (\ref{eqn:whatever123}) that transform ${\mathcal P}$ into ${\mathcal P}'$.  Let 
$\mathcal R={\mathcal P}_{h}$, where $h$ corresponds to the first index in ${\tt NE}$, i.e.,
$h=\#{\tt SAME}+\#{\tt SW}+1$.

   $\mathcal R$ is obtained from $\mathcal P$ by applying a series of southwest long moves and $\mathcal R<\mathcal P$.  Likewise, $\mathcal P'$ is obtained from $\mathcal R$ entirely by northeast long moves, so $\mathcal R<\mathcal P'$.
Suppose we have some other $<$-lower bound $\mathcal S$ of $\mathcal P$ and $\mathcal P'$.  We may construct $\mathcal T$, 
a $<$-lower bound for $\mathcal R$ and $\mathcal S$, by the same argument above we have used to construct ${\mathcal R}$ from 
${\mathcal P}$ and ${\mathcal P}'$. 

That $\mathcal R=\mathcal P \wedge{\mathcal P}'$ is immediate from the following:

\begin{claim}
\label{claim:diagpf}
$\mathcal R\geq \mathcal T=\mathcal S$. 
\end{claim}
\noindent
\emph{Proof of Claim~\ref{claim:diagpf}:}
We will work with the sets: 
\[A=\{a: +_a \text{\ in ${\mathcal R}$ is strictly southwest of $+_a$ in $\mathcal P$}\}\]
\[A'=\{a: +_a \text{\ in ${\mathcal R}$ is strictly southwest of $+_a$ in $\mathcal P'$}\}\]
\[B=\{a: +_a \text{\ in ${\mathcal S}$ is strictly southwest of $+_a$ in $\mathcal P$}\}\]
\[B'=\{a: +_a \text{\ in ${\mathcal S}$ is strictly southwest of $+_a$ in $\mathcal P'$}\}\]
\[C=\{a: +_a \text{\ in ${\mathcal T}$ is strictly southwest of $+_a$ in $\mathcal R$}\}\]
\[C'=\{a: +_a \text{\ in ${\mathcal T}$ is strictly southwest of $+_a$ in $\mathcal S$}\}\]
Summarizing, we have:
\begin{equation}\nonumber
\label{eqn:schematic}
\xygraph{
!{<0cm,0cm>;<1cm,0cm>:<0cm,1cm>::}
!{(0.75,0) }*+{\mathcal P}="p"
!{(2.25,0) }*+{\mathcal P'}="q"
!{(0,-1.25) }*+{\mathcal R}="r1"
!{(3,-1.25) }*+{\mathcal S}="r2"
!{(1.5,-2) }*+{\mathcal T}="r"
"p"-"r1"_{A}
"p"-"r2"_(0.7){B}
"q"-"r1"^(0.7){A'}
"q"-"r2"^{B'}
"r1"-"r"_{C}
"r2"-"r"^{C'}
}
\end{equation}

Since $\mathcal T<\mathcal R<\mathcal P$ and $\mathcal T<\mathcal S<\mathcal P$, 
\begin{equation}
\label{eqn:thefirst567}
A\cup C=\{a:+_a \text{ in } \mathcal T \text{ is strictly southwest of } +_a \text{ in } \mathcal P\}=B\cup C'.
\end{equation}
Similarly,
\begin{equation}
\label{eqn:thesecond567}
A'\cup C=B'\cup C'.
\end{equation}
  By the construction of ${\mathcal R}$ and ${\mathcal T}$ we have 
\begin{equation}
\label{eqn:blah878}
C\cap C'=\emptyset
\end{equation}
and
\begin{equation}
\label{eqn:blah999}
A\cap A'=\emptyset.
\end{equation} 
Intersecting both sides of (\ref{eqn:thefirst567}) by $C'$ gives:
\[(A\cup C)\cap C'=(B\cup C')\cap C' \iff (A\cap C')\cup(C\cap C')=(B\cap C')\cup(C'\cap C').\]
By (\ref{eqn:blah878}) we have $A\cap C'=C'$, which in turn implies $C'\subseteq A$.  
Likewise, (\ref{eqn:thesecond567}) implies $C'\subseteq A'$.  Therefore by (\ref{eqn:blah999}), $C'=\emptyset$ holds. This shows $\mathcal R\geq \mathcal T=\mathcal S$, as claimed. \qed

This completes the proof of the proposition.
\qed

\begin{lemma}
\label{lemma:meetporism}
Fix $\mathcal P,\mathcal P'\in {\tt MinPlus}(X_u')$ (where $u$ is biGrassmannian), and let \[\mathcal P=:\mathcal P_1\mapsto \mathcal P_2\mapsto \ldots 
\mapsto \mathcal P_h \mapsto \mathcal P_{h+1}\mapsto
\ldots
\mapsto \mathcal P_{\ell}\mapsto \mathcal P_{\ell +1}:=\mathcal P'\] be the sequence (\ref{eqn:whatever123}) in Proposition~\ref{prop:moveorder} (II).  Then for any $1\leq h\leq \ell+1$, 
${\mathcal P}_h\subseteq {\mathcal P}\cup {\mathcal P}'$. In particular, 
if ${\mathcal R}={\mathcal P}\wedge{\mathcal P}'$ in
the lattice $({\tt MinPlus}(X_u'),<)$ (where $u$ is biGrassmannian), then
${\mathcal R}\subseteq {\mathcal P}\cup {\mathcal P}'$.
\end{lemma}
\begin{proof}
The claim about ${\mathcal P}_h$ follows from the construction of the sequence (\ref{eqn:whatever123}) in
Proposition~\ref{prop:moveorder}(II). The claim about ${\mathcal R}$ holds since in the proof of
Proposition~\ref{prop:moveorder}(II), we have shown ${\mathcal R}=\mathcal P_h$ (where $h=\#{\tt SAME}+\#{\tt SW}+1$).
\end{proof}

\subsection{Proof of Theorem~\ref{theorem:overlay}}
(I): Fix $w\in S_n$, $\mathcal P\in {\tt MinPlus}(X_w')$.  Let 
\[\mathcal Q=(\mathcal P_1 , \mathcal P_2, \dots , \mathcal P_k), \mathcal Q'=(\mathcal P'_1 ,\mathcal P'_2, \dots , \mathcal P'_k)
\in {\tt Multi}({\mathcal P}).\] 
By Proposition~\ref{prop:moveorder}, we may connect $\mathcal P_1$ to $\mathcal P_1'$ by long moves, as in (\ref{eqn:whatever123}):  
\[\mathcal P_1=\mathcal P_{1,1}\mapsto \mathcal P_{1,2}\mapsto \cdots 
{\mathcal P}_{1,h}\mapsto {\mathcal P}_{1,h+1}
\mapsto \cdots \mapsto \mathcal P_{1,\ell+1}=\mathcal P_1'.\] 
By Lemma~\ref{lemma:meetporism}, $\mathcal P_{1,h},\subseteq \mathcal P_1\cup \mathcal P_1'\subseteq \mathcal P$. Hence
\[{\tt supp}(\mathcal P_{1,h},\mathcal P_2,\cdots, \mathcal P_k)\subseteq\mathcal P.\]  
Since ${\mathcal P}\in {\tt MinPlus}(X_w')$, this containment is an
equality.  That is, each $(\mathcal P_{1,h},\mathcal P_2,\cdots \mathcal P_k)\in {\tt Multi}({\mathcal P})$. In the case $h=\ell+1$, we reach $(\mathcal P'_1 , \mathcal P_2, \mathcal P_3, \ldots , \mathcal P_k)$ from $(\mathcal P_1 , \mathcal P_2,  \mathcal P_3, \ldots , \mathcal P_k)$.

Continuing in this manner, one connects  
\[(\mathcal P'_1 , \mathcal P_2, \mathcal P_3, \ldots , \mathcal P_k) \mbox{\ to \ } (\mathcal P'_1 , \mathcal P'_2, \mathcal P_3, \ldots , \mathcal P_k),\] 
by long moves, that keep one in ${\tt Multi}({\mathcal P})$, until one reaches  $(\mathcal P'_1 , \mathcal P'_2, \ldots , \mathcal P'_k)$.

(II): This is Proposition~\ref{prop:moveorder}(II).

(III): 
  Let  $\mathcal Q,\mathcal Q'\in {\tt Multi}(\mathcal P)$, as in (I).  Since by definition ${\tt Multi}(\mathcal P)$ is a subposet of the lattice ${\tt Multi}(w)$, it suffices to show  $\mathcal Q\wedge \mathcal Q'\in {\tt Multi}(\mathcal P)$.  (The argument for the join is similar.)
 By Proposition~\ref{prop:moveorder}, for each $i$, 
there is $\mathcal R_i=\mathcal P_i\wedge\mathcal P_i'$. In general, the meet in a Cartesian product of lattices is formed by taking the meet in each component. Therefore, 
\[{\mathcal Q}\wedge{\mathcal Q}'=(\mathcal R_1, \mathcal R_2 , \ldots , \mathcal R_k)\in {\tt Multi}(w).\] 

 By Lemma~\ref{lemma:meetporism},
$\mathcal R_i\subseteq \mathcal P_i\cup \mathcal P_i' \subseteq \mathcal P$.  Hence
${\tt supp}({\mathcal Q}\wedge{\mathcal Q}')\subseteq \mathcal P$. However, since $\mathcal P\in {\tt MinPlus}(X_w')$, 
we must have ${\tt supp}({\mathcal Q}\wedge{\mathcal Q}')=\mathcal P$, i.e., 
$\mathcal Q\wedge\mathcal Q'\in {\tt Multi}(\mathcal P)$, as desired.  \qed

The ``{\tt MinPlus}'' hypothesis of Theorem~\ref{theorem:overlay} is necessary, as we now demonstrate:
\begin{example}[${\tt Multi}(\mathcal P)$ for non minimal plus diagrams]
\label{ex:notlattice}
Let $w=14253$.  Then we have
${\tt biGrass}(w)=\{{\color{blue}14235},{\color{Plum}12453}\}$.  Let 
\[\mathcal P=\left [ \begin{array}{ccccc}
\cdot & \cdot & +& \cdot &\cdot\\ 
+ & + & \cdot & \cdot &\cdot\\ 
\cdot & + & \cdot & \cdot &\cdot\\ 
\cdot & \cdot & \cdot & \cdot &\cdot\\ 
\cdot & \cdot & \cdot & \cdot &\cdot\\ 
\end{array}\right]\in {\tt Plus}(X_w')\setminus  {\tt MinPlus}(X_w').\]
Observe
\[{\tt Multi}(\mathcal P)=\left\{
\left [ \begin{array}{ccccc}
\cdot & \cdot & \color{Plum}+ & \cdot &\cdot\\ 
\color{blue}+ & \color{blue}+ & \cdot & \cdot &\cdot\\ 
\cdot & \color{Plum}+ & \cdot & \cdot &\cdot\\ 
\cdot & \cdot & \cdot & \cdot &\cdot\\ 
\cdot & \cdot & \cdot & \cdot &\cdot\\ 
\end{array}\right],
\left [ \begin{array}{ccccc}
\cdot & \cdot & \color{blue}+ & \cdot &\cdot\\ 
\color{blue}+ & \color{Plum}+ & \cdot & \cdot &\cdot\\ 
\cdot& \color{Plum}+ & \cdot & \cdot &\cdot\\ 
\cdot & \cdot & \cdot & \cdot &\cdot\\ 
\cdot & \cdot & \cdot & \cdot &\cdot\\ 
\end{array}\right]
\right\}.\]
${\tt Multi}(\mathcal P)$ consists of two incomparable elements, so is in particular not a lattice. \qed
 \end{example}

\subsection{Conclusion of the proof of Theorem~\ref{theorem:main}}
By Lemma~\ref{lemma:mainobs}(II), ${\tt Multi}(\mathcal P)\neq \emptyset$. In addition, by Theorem~\ref{theorem:overlay}, ${\tt Multi}(\mathcal P)$ is a finite lattice and thus has a unique minimum
${\mathcal M}_{\mathcal P}$. Let 
\[\overline{{\tt Multi}(w)}:=\{{\mathcal M}_{{\mathcal P}}:\mathcal P\in {\tt MinPlus}(X'_w)\}.\]  
Hence, trivially, we have a bijection 
\[\Psi:\overline{{\tt Multi}(w)} \to {\tt MinPlus}(X_w').\] 

Let ${\tt AllPrism}(w)$ denote the set of  \emph{all} prism tableaux and ${\tt MinPrism}(w)$ the set of minimal prism tableaux for $w$. From the definitions, 
\[{\tt Prism}(w)\subseteq {\tt MinPrism}(w)\subseteq {\tt AllPrism}(w).\]  


\begin{claim}
\label{claim:Phithing}
There is a bijection $\Phi:{\tt AllPrism}(w)\rightarrow {\tt Multi}(w)$. 
\end{claim}
\begin{proof} We associate each
${\mathcal P}\in {\tt MinPlus}(X_{u_e}')$ with a filling of $R_e$. To do this, notice that by definition
$R_{e}$ sits in $n\times n$ exactly as the $+$'s of ${\mathcal D}_{\tt bot}(u_e)$ do. Hence by Lemma~\ref{lemma:localmoves}(II)
there is a bijection between the $+$'s of ${\mathcal P}$ and the boxes of $R_{e}$. 

Assign to each box of 
$R_e$ the colored label that is the row position of that boxes' associated $+$ in ${\mathcal P}$. Then these
labels of color $e$ satisfy (S1) by definition. That they satisfy (S2) and (S3) follows from Lemma~\ref{lemma:weaklysouthwest}. Finally, (S4) holds by Lemma~\ref{lemma:localmoves}(II). 

The map we have just described from ${\mathcal P}\in {\tt MinPlus}(X_{u_e}')$ and the (S1)-(S4) fillings ${\mathcal S}$ of $R_e$ is clearly injective. That it is a surjection 
follows since the tableaux ${\mathcal S}$ are clearly in weight-preserving bijection with the semistandard 
tableaux for ${\mathfrak S}_{u_e}$ \cite[Proposition~2.6.8]{Manivel}, which are known to be in bijection with
${\mathcal P}\in {\tt MinPlus}(X_{u_e}')$, see, e.g., \cite[Proposition~5.3]{KMY} (and for an earlier reference,
see \cite{Kogan}).  

Now, given $T\in {\tt AllPrism}(w)$ we
construct $\Phi(T):=({\mathcal P}_1,\ldots, {\mathcal P}_k)\in {\tt Multi}(w)$ 
by applying the above correspondence independently to each $R_e$. That this
is a bijective map follows from the bijectivity on each component.
\end{proof}

\begin{corollary}
\label{cor:Phirestricts}
$\Phi$ restricts to a bijection  
$\widetilde{\Phi}:{\tt MinPrism}(w) \rightarrow {\tt supp}^{-1}({\tt MinPlus}(X_w'))\subseteq {\tt Multi}(w)$. 
\end{corollary}
\begin{proof}
Since $\Phi$ is a bijection, we are only required to show that
\[{\rm im} \ \Phi|_{{\tt MinPrism}(w)}={\tt supp}^{-1}({\tt MinPlus}(X_w')).\] 
However,
this holds, since
a tableau $T\in{\tt AllPrism}(w)$ is in 
${\tt MinPrism}(w)$ if and only if ${\tt supp}(\Phi(T))$ has cardinality $\ell(w)$, i.e. if and only if ${\tt supp}(\Phi(T))\in {\tt MinPlus}(X_w')$. 
\end{proof}
 
\begin{claim}
\label{claim:longMove}
Let $\mathcal Q\in {\tt Multi}(w)$ and let $T=\Phi^{-1}(Q)$.  Then $T$ has an unstable triple  if and only if  there exists a southwest long move $\mathcal Q\mapsto \mathcal Q'$ such that ${\tt supp}(\mathcal Q)={\tt supp}(\mathcal Q')$.
\end{claim}
\noindent
\emph{Proof of Claim~\ref{claim:longMove}:}
($\Rightarrow$) Suppose $T$ has an unstable triple $\{\ell_c,\ell_d,\ell'_e\}$ contained in an antidiagonal $D$.  Let $T'$ be the tableau obtained by replacing $\ell_c$ with $\ell_c'$.  We must show $\mathcal Q$ differs from $\mathcal Q':=\Phi(T')$ by a southwest long move such that ${\tt supp}({\mathcal Q})={\tt supp}({\mathcal Q}')$.  If one could not conduct a southwest long move, there must have been some $+$ of color $c$ in the region consisting of:
\begin{enumerate}
\item rows $\ell,\ell+1,\ldots, \ell'-1$ of $D_{\rm left}$. 
\item rows $\ell+1,\ell+2,\ldots, \ell'$ of $D_{\rm right}$.
\item  rows $\ell+1,\ell+2,\ldots, \ell'$ of $D$.
\end{enumerate}
    Moving the $+$ of color $c$ in row $\ell$ to row $\ell'$ would cause it to appear southwest of $+_b$, contradicting Lemma~\ref{lemma:relativeorder}.

So now assume we have a southwest long move. It remains to check the
support assertion.
The labels $\ell_c$ and $\ell_d$ in $T$ each ensure there is a $+$ in row $\ell$ of $D$ in ${\tt supp}(\mathcal Q)$, while $\ell'_e$ gives a $+$ to row 
$\ell'$ of $D$ in ${\tt supp}(\mathcal Q)$.  Similarly,  $\ell_d$ in $T'$ corresponds to a plus in row $\ell$ of $D$ in ${\tt supp}(\mathcal Q')$, while $\ell'_c$ and $\ell'_e$ gives each ensure there is a $+$ in row $\ell'$ of $D$ in ${\tt supp}(\mathcal Q')$.  So replacing $\ell_c$ in $T$ with $\ell_c'$ in $T'$ gives ${\tt supp}(\mathcal Q)={\tt supp}(\mathcal Q')$. 

($\Leftarrow$) Suppose we may apply a support preserving southwest long move to 
\[\mathcal Q=(\mathcal P_1,\ldots, \mathcal P_c, \ldots, \mathcal P_k)\mapsto \mathcal Q'=(\mathcal P_1,\ldots,\mathcal P_c',\ldots ,\mathcal P_k).\]  
That is, there is an antidiagonal $D\subset n\times n$ such that $\mathcal P_c$ contains a $+$ in row $\ell$ of $D$ that may be moved to row $\ell'>\ell$ by a southwest long move.  Since ${\tt supp}(\mathcal Q)={\tt supp}(\mathcal Q')$, there must be colors $d,e$ with the property that $\mathcal P_d$ has a $+$ in row $\ell$ of $D$ and $\mathcal P_e$ has a $+$ is row $\ell'>\ell$ of $D$.  In $T$, this implies that there are labels $\{\ell_c$, $\ell_d$, $\ell'_e\}$ in $D$.  Let $T'$ be obtained from $T$ by replacing $\ell_c$ with $\ell'_c$.  Then $T'=\Phi(\mathcal Q)\in {\tt AllPrism}(w)$. So $\{\ell_c,\ell_d,\ell'_e\}$ is an unstable triple.
\qed

\begin{claim}
$\Phi$ (further) restricts to a bijection, $\widehat \Phi:{\tt Prism}(w)\to \overline{{\tt Multi}(w)}$.
\label{claim:prismBijection}
\end{claim}
\begin{proof}
Since we know $\Phi$ is a bijection, we need only show that ${\rm im} \ \Phi|_{{\tt Prism}(w)}=\overline{{\tt Multi}(w)}$.

If $\mathcal M_{\mathcal P}\in \overline{{\tt Multi}(w)}$, then $\mathcal M_{\mathcal P}$ is by definition the minimum in ${\tt Multi}(\mathcal P)$. Let $T:={\widetilde \Phi}^{-1}(\mathcal M_{\mathcal P})$ (this exists by Corollary~\ref{cor:Phirestricts}).  By Claim~\ref{claim:longMove} $(\Rightarrow)$, if  $T$ has an unstable triple, then there exists a southwest long move $\mathcal M_{\mathcal P} \mapsto \mathcal Q'$ so that ${\tt supp}(\mathcal M_{\mathcal P})={\tt supp}(\mathcal Q')$.  But then $\mathcal Q'<'\mathcal M_\mathcal P$, a contradiction. Hence $T\in {\tt Prism}(w)$. Thus,
${\rm im} \ \Phi|_{{\tt Prism}(w)}\supseteq\overline{{\tt Multi}(w)}$.

Suppose $T\in {\tt Prism}(w)\subseteq {\tt MinPrism}(w)$, and let $\mathcal Q:=\widetilde\Phi(T)\in  {\tt supp}^{-1}({\tt MinPlus}(X_w'))$. Suppose $\mathcal Q$ is not the minimum element in ${\tt Multi}({\tt supp}(\mathcal Q))$.  Then there exists a southwest long move $\mathcal Q\mapsto \mathcal Q'$, so that ${\tt supp}(\mathcal Q)={\tt supp}(\mathcal Q')$.  Then by Claim~\ref{claim:longMove} $(\Leftarrow)$,  $T$ must have had an unstable triple, contradicting $T\in {\tt Prism}(w)$.  Thus, $\mathcal Q=\mathcal M_{{\tt supp}(\mathcal Q)}\in \overline{{\tt Multi}(w)}$. This shows
${\rm im} \ \Phi|_{{\tt Prism}(w)}\subseteq\overline{{\tt Multi}(w)}$, as required.
\end{proof}

By Claim~\ref{claim:prismBijection}, $\Psi\circ \widehat \Phi:{\tt Prism}(w)\to {\tt MinPlus}(X_w')$ is a bijection. Now,  
  \[{\tt wt}(T)=\prod_{i}x_i^{\text{$\#$ of antidiagonals containing $i$}}
\text{ \ and \ 
${\tt wt}(\mathcal P)=\prod_{i}x_i^{\text{$\#$ of $+$'s in row $i$}}$.}\]
That ${\tt wt}(T)={\tt wt}((\Psi\circ\widehat\Phi)(T))$ is immediate
from these definitions.  Hence the theorem follows.\qed

\section{Further discussion}

\subsection{Comparisons to the literature}
Ultimately, the evaluation of any model for Schubert polynomial rests on
its success towards the \emph{Schubert problem}, i.e., finding a generalized
Littlewood-Richardson rule for Schubert polynomials. Due to the analogy
with {\sf Sym}, one hopes that a solution will not only provide \emph{merely} a rule, but rather lead to an entire companion combinatorial theory. This would 
presumably enrich our understanding of {\sf Pol} and its role in mathematics 
just as the Young tableau theory does for {\sf Sym}. 

That the prism model manifestly uses Young tableaux is our impetus for ongoing investigations that fundamental tableaux algorithms might
admit prism-generalizations.

The first rule for Schubert polynomials was conjectured by \cite{Kohnert}. This rule begins with the diagram of $w$ and evolves other subsets of $n\times n$ by a simple move, the Schubert polynomial is a generating series over these subsets. 
A proof is presented in \cite{Winkel, Winkel:again}.  Arguably, this rule
is the most handy of all known rules, even though the set of Kohnert diagrams does not have a closed description.

Probably the most well-known and utilized formula is given by \cite{BJS}, which expresses the Schubert polynomial 
in terms of reduced decompositions of $w$. This rule is made graphical by the $RC$-graphs of \cite{Bergeron.Billey} (cf. \cite{Fomin.Kirillov}). One can obtain any $RC$-graph for $w$ from any other by the \emph{chute} and \emph{ladder} moves of \cite{Bergeron.Billey}.  


While neither of the above rules transparently reduces to
the tableau rule for Schur polynomials,  it is not too difficult to show in either case, that the objects involved do biject with
semistandard tableaux, see \cite{Kohnert} and \cite{Kogan}
respectively.

We are not aware of any published bijection between the Kohnert rule and any other model for Schubert polynomials. On the other
hand, there is a map between the prism tableaux and
$RC$-graphs: the labels on the $i$-th antidiagonal
indicate the row position of the $+$'s on the same antidiagonal in the associated $RC$-graph. This map is
clearly injective but we do not currently have a purely \emph{combinatorial}
proof that the map is well-defined. Tracing our proof of the main theorem, well-definedness comes from the Gr\"{o}bner basis theorem of \cite{Knutson.Miller:annals}. Moreover, in said proof, we treat each
$RC$-graph as a \emph{specific} overlay of $RC$-graphs for bigrassmannian
permutations. The latter $RC$-graphs are in bijection with semistandard
tableaux of rectangular shape. This is the reason for the ``dispersion'' remark of the introduction.

The work of \cite{balanced} gives a tableau rule for Schubert polynomials of a different flavor. This rule treats ${\mathfrak S}_w$ as a generating series for {\bf balanced fillings} of the diagram of $w$. The reduction to semistandard tableaux for Grassmannian $w$ seems non-trivial.

In \cite{BKTY}, a formula is given for a Schubert polynomial
as a nonnegative integer linear combination of sum of products of Schur functions in disjoint sets of variables (with nontrivial coefficients). 
This is also in some sense a tableau formula for ${\mathfrak S}_w$. 
In \cite{Lenart:review} this result is rederived as a consequence of the crystal
graph structure on $RC$-graphs developed there.

\subsection{Details of the reduction to semistandard tableaux}
We now explicate the reduction from prism tableaux to ordinary semistandard tableaux, as indicated in the introduction.

\excise{Suppose $w\in S_n$ is Grassmannian.  Then $\mathfrak S_w$ is a Schur polynomial, and each minimal prism tableau reduces to a semistandard tableau by eliminating redundant labels, flipping the diagram vertically, and replacing $i\mapsto k-i+1$, where $k$ is the position of the unique descent of $w$.}

\begin{proposition}
\label{prop:grassredux}
Assume $w\in S_n$ is Grassmannian.

(I) The shape $\lambda(w)$ is a Young diagram, in French notation.

(II) Let $T\in {\tt MinPrism}(w)$. All labels of a box of $T$ have the same number.

(III) $T$ does not have unstable triples, i.e., ${\tt MinPrism}(w)={\tt Prism}(w)$.
\end{proposition}

\begin{proof}
(I): Since $w$ is Grassmannian, it has a unique descent, $w(k)>w(k+1)$.  Furthermore, all essential boxes of $w$ lie in the $k$th row, say in columns $a_1<\ldots< a_j$.  The rectangle $R_{(k,a_i)}$ starts at row $r_w(k,a_i)+1$, which strictly increases as $i$ increases, since essential boxes to the right in the diagram take on higher values for the rank function.  Each rectangle is left justified by construction, and has width $a_i-r_w(k,a_i)$.  This value strictly increases with each $i$, since 
\[a_{i}+(r_w(k,a_{i+1})-r_w(k,a_i))<a_{i+1},\] 
as seen from the diagram of $w$.  So the rectangles $R_{(k,a_i)}$ overlap to form the shape of a partition.

(II): Suppose not. Let ${\sf x}$ be a ``bad'' box, i.e., one with $\ell_c$ and $\ell'_d$ in ${\sf x}$ where $\ell\neq \ell'$ (and thus $c\neq d$). We may assume ${\sf x}$ is the northeast-most bad box. Let $D$ be the antidiagonal containing ${\sf x}$. We may also assume that each box of $D$ contains
a label of color $c$. 

\noindent
{\sf Case 1:} ($\ell'>\ell$): Then by (S2) and (S3) the labels of color $c$ in 
$D$, that are 
strictly northeast of ${\sf x}$, are all distinct and different than both $\ell$ and $\ell'$. For the same reason, all labels of
color $d$ in $D$ strictly southwest of ${\sf x}$ are distinct and different than 
$\ell$ and $\ell'$. Hence the total number of distinct numbers in $D$ exceeds $\#D$.
Since $|\lambda(w)|=\ell(w)$, we conclude $T$ cannot be minimal, a contradiction.

\noindent
{\sf Case 2:} ($\ell'<\ell$): Again by (S2) and (S3), all labels of color $c$ that are in $D$
but strictly southwest of ${\sf x}$ are distinct and are also different than $\ell$ and $\ell'$. Hence
if $D$ is not overfull (i.e., has more distinct numbers than boxes) $\ell'_c$ must appear in a box ${\sf y}$ in 
$D$ that is strictly northeast of ${\sf x}$. By the definition of prism tableaux and $\lambda(w)$, either
there exists: 
\begin{enumerate}
\item a box ${\sf z}$ in the row of ${\sf x}$
and in the column of ${\sf y}$ that contains 
labels $m_c$ and $m'_{d}$, or 
\item a box ${\sf w}$ in the column of ${\sf x}$
and in the row of ${\sf y}$ that contains 
labels $m_c$ and $m'_{d}$.
\end{enumerate}
We may assume the first case occurs (the argument for the other case is the same).
In view of
the $\ell'_c\in {\sf y}$ combined with (S3), $m>\ell'$. On the other hand, in view of the
$\ell'_d\in {\sf x}$ combined with (S2), $m'\leq \ell'$. Hence we see ${\sf z}$ is a bad box strictly
east of ${\sf x}$, a contradiction of the extremality of ${\sf x}$. 

(III): Suppose $T$ has an unstable triple $\{\ell_c,\ell_d,\ell'_e\}$ in antidiagonal $D$.  Let $T'$ be the tableaux obtained by replacing $\ell_c$ with $\ell'_c$.  Then by definition, $T'\in {\tt Prism}(w)$.  By (II), $\ell_c$ must sit in a box containing no other labels.  By the definition of $\lambda(w)$, 
this furthermore implies every box of $D$ in $\lambda(w)$ 
has a label of color $c$. 
(II) then implies 
the box containing $\ell'_e$ must contain a label $\ell'_c$.  This contradicts (S2) and (S3) combined.
 \excise{ If $T$ has an unstable triple, we are able to change a label $\ell$ of color $c$ to a different value.  By (II), this can only happen in a box for which $c$ is the only color.  So the color $c$ extends along the whole antidiagonal.  By the definition of an unstable triples, we have $\ell<\ell'$ of colors $d$ and $d'$ respectively, sitting in the same antidiagonal such that replacing the $\ell$ of color $c$ with $\ell'$ also gives a prism tableau.  However, the box containing the $\ell'$ of color $d'$ must also have a label of color $c$, which takes the value $\ell'$, again by (II).  So such a replacement violates reverse semistandardness, giving a contradiction.}
\end{proof}

\subsection{Stable Schubert polynomials}
The {\bf stable Schubert polynomial} (also known as the
\emph{Stanley symmetric polynomial}) is the generating series 
defined by 
\[F_w(x_1,x_2,\ldots):=\lim_{m\to \infty} {\mathfrak S}_{1^m\times w},\]
where if $w\in S_n$ then $1^m\times w$ is the permutation in $S_{m+n}$ defined by\[\text{$(1^m\times w)(i)=i$ for 
$1\leq i\leq m$ and $(1^m\times w)(m+i)=m+w(i)$ for $1\leq i\leq n$.}\]

It is true that 
\[F_w(x_1,x_2,\ldots,x_m,0,0,\ldots)={\mathfrak S}_{1^m\times w}(x_1,\ldots,x_m,0,0,\ldots).\] 
Now,
notice that $\lambda(1^m\times w)$ and $\lambda(w)$ are the same shape, but the former is shifted down $m$ steps in the grid
relative to $\lambda(w)$. Therefore it follows that \[F_w(x_1,x_2,\ldots,x_m,0,0,\ldots)=\sum_{T}{\tt wt}(T),\] 
where the
sum is over all \emph{unflagged} (i.e., exclude (S4)) minimal prism tableaux of shape $\lambda(w)$ that use the labels
$1,2,\ldots,m$. In the limit, this argument implies the generating series $F_w(x_1,x_2,\ldots)$ is given by the same formula, except we
allow all labels from ${\mathbb N}$.

\subsection{An overlay interpretation of chute and ladder moves}
In \cite{Bergeron.Billey}, {\bf chute moves} were defined for pipe dreams. These moves are locally of the form
\begin{equation}\label{eqn:chute}
{\mathcal P}=\begin{matrix}
. & + & + & + & \cdots & + & + & .\\
+ & + & + & + & \cdots & + & + & .
\end{matrix}\ \ \ 
\to \ \ \ 
\begin{matrix}
\cdot & + & + & + & \cdots & + & + & +\\
\cdot & + & + & + & \cdots & + & + & \cdot
\end{matrix}={\mathcal Q}
\end{equation}
Suppose $\mathcal P\in {\tt MinPlus}(X_w')$, ${\tt biGrass}(w)=\{u_1,\ldots,u_k\}$ and ${\mathcal P}={\mathcal P}_1\cup\cdots\cup {\mathcal P}_k$, where $\mathcal P_i\in {\tt MinPlus}(X_{u_i}')$. We now show: 
\begin{quotation}
\emph{The chute move's ``long jump'' of a single $+$ may be interpreted as a sequence of the northeast local moves
(\ref{local1}) applied to the ${\mathcal P}_i$'s.}
\end{quotation}

\begin{example}
\label{exa:overlaychute}
 Let $w=1432$. Now, 
${\tt biGrass}(w)=\{u_1=1423, u_2=1342\}$. Consider the following sequence of northeast moves
\[\left[\begin{array}{cccc} 
\cdot&\color{Plum}+&\cdot &\cdot\\
\color{blue}+ &\color{blue}+ \color{Plum}+&\cdot&\cdot\\
\cdot&\cdot&\cdot&\cdot\\
\cdot&\cdot&\cdot&\cdot\\
\end{array}\right]
\rightarrow
\left[\begin{array}{cccc} 
\cdot&\color{Plum}+&\color{blue}+ &\cdot\\
\color{blue}+ &\color{Plum}+&\cdot&\cdot\\
\cdot&\cdot&\cdot&\cdot\\
\cdot&\cdot&\cdot&\cdot\\
\end{array}\right]
\rightarrow
\left[\begin{array}{cccc} 
\cdot&\color{blue}+\color{Plum}+&\color{blue}+&\cdot\\
\cdot &\color{Plum}+&\cdot&\cdot\\
\cdot&\cdot&\cdot&\cdot\\
\cdot&\cdot&\cdot&\cdot\\
\end{array}\right]
\]
Let the support of the first and third plus diagrams be ${\mathcal P}$ and ${\mathcal Q}$, respectively.
We have ${\mathcal P},{\mathcal Q} \in {\tt MinPlus}(X_{w}')$. ${\mathcal P}$ and ${\mathcal Q}$ differ by a chute move.
At the level of the overlays, one sees this transition as an application of (\ref{local1}) to each blue $+$ in the second row.\qed  
\end{example}

Example~\ref{exa:overlaychute} indicates the general pattern.
Let $(i,j)$ be the position of the southwest $+$ of ${\mathcal P}$ in (\ref{eqn:chute}) and $(i-1,j')$ the position of the northeast $+$ of ${\mathcal Q}$ in (\ref{eqn:chute}).  Without loss of generality,
we may assume each ${\mathcal P}_1,\ldots,{\mathcal P}_t$ contains a $+$ at $(i,j)$ while all other
${\mathcal P}_h$ do not.

\begin{claim}
\label{claim:chuteA}
Consider the interval of consecutive $+$'s in row $i$ of ${\mathcal P}_h$ ($1\leq h\leq t$)
starting at the left with the $+$ in position $(i,j)$. One can apply the move 
(\ref{local2}) (in the right to left order) to each of these $+$'s to obtain ${\mathcal P}_1',\ldots,{\mathcal P}_t'$.  
\end{claim}
\begin{proof}
It follows from Lemma~\ref{lemma:relativeorder} that the configurations
given below do not appear in ${\tt MinPlus}(X_u')$ whenever $u$ is biGrassmannian:
\[\left\{ \left[
\begin{matrix}\cdot&+\\+&+\end{matrix}\right], \left[\begin{matrix}\cdot &+\\+&\cdot\end{matrix}\right],\left[\begin{matrix}+&+\\+&\cdot\end{matrix}\right]
\right\}.\]
Suppose there is an obstruction to one of the local moves.  It must appear in row $i-1$.  Due to the $+$ in position $(i,j)$, such an obstruction necessarily forces one of the above configurations to appear, causing a contradiction.\end{proof}

\begin{claim}
\label{claim:chuteContain}
$\mathcal P_h'\subseteq \mathcal Q$ for $1\leq h \leq t$ and $\mathcal P_h\subseteq \mathcal Q$ for $t+1\leq h\leq k$.
\end{claim}
\begin{proof}
First suppose $1\leq h\leq t$. Each move from Claim~\ref{claim:chuteA} takes a $+$ from position $(i,a)$ with $j\leq a < j'$ and replaces it with a $+$ in position $(i-1,a+1)\in \mathcal Q$.  Furthermore, each $\mathcal P_h'$ has no $+$ in position $(i,j)$.  So $\mathcal P_h'\subseteq \mathcal Q$. If $h\geq t+1$, then by assumption $\mathcal P_h$ has no $+$ in position $(i,j)$.  So $\mathcal P_h\subseteq \mathcal P\backslash \{(i,j)\}\subseteq \mathcal Q$.
\end{proof}

\begin{claim}
\label{claim:chuteB} 
${\mathcal Q}={\mathcal P}_1'\cup \cdots \cup {\mathcal P}_t'\cup 
{\mathcal P}_{t+1}\cup \cdots \cup {\mathcal P}_k$.
\end{claim}
\begin{proof}
Let $\widetilde{\mathcal Q}= \mathcal P_1'\cup \ldots \mathcal P_t'\cup \mathcal P_{t+1}\ldots \cup \mathcal P_k$.  Suppose $\widetilde{\mathcal Q}\neq \mathcal Q$.  By Claim~\ref{claim:chuteContain}, each $\mathcal P_i'\subseteq \mathcal Q$, for $1\leq i\leq t$ and $\mathcal P_i\subseteq \mathcal Q$ for $t+1\leq i\leq k$.  Then $ \mathcal Q\supsetneq \widetilde{\mathcal Q}\in {\tt Plus}(X_w')$, contradicting the assumption that $\mathcal Q\in {\tt MinPlus}(X_w')$. 
\end{proof}

 A similar discussion applies to the ladder moves.

\subsection{Future work}

It is straightforward to assign weights to prism tableau in order to give a formula for double Schubert polynomials.

A generalization to Grothendieck polynomials requires a deeper
control of the overlay procedure. In investigating this, one is led to some results of possibly independent interest.

Specifically, for Theorem~\ref{theorem:main}, we have used the fact that the facets of 
$\Delta_{X_w'}$ are intersections of facets of those associated to ${\tt biGrass}(w)$.  One can make a similar conjecture for all \emph{interior} faces $w$'s complex. Each $\Delta_{X_w'}$ is a ball or sphere \cite[Theorem 3.7]{Knutson.Miller:Advances}. Hence one can refer to the interior faces of
this complex. Let \[{\tt IntPlus}(w)=\{\mathcal P:\mathcal P\in {\tt Plus}(w) \text{ and } \mathcal F_{\mathcal P} \text{ is an interior face of } \Delta_{X_w'}\}.\]
\begin{conjecture} 
\label{conj:theconj}
${\tt IntPlus}(w)\subseteq\{\mathcal P_1\cup\cdots \cup \mathcal P_k: \mathcal P_i\in {\tt IntPlus}(u_i), \text{ for } u_i\in {\tt biGrass}(w)\}$.
\end{conjecture}
Conjecture~\ref{conj:theconj} has been exhaustively 
computer checked for all $n\leq 6$. 

As part of an intended proof of Conjecture~\ref{conj:theconj}, 
one defines $K$-theoretic analogues of the chute and ladder moves
of \cite{Bergeron.Billey}: that is if $\mathcal P\to \mathcal Q$ by a chute move (respectively, ladder move) then $\mathcal P\to \mathcal P\cup \mathcal Q$ is a $K$-chute (respectively,
$K$-ladder move). Whereas not all interior plus diagrams 
are connected by the original chute and ladder moves, it is true that 
they are connected once one allows the extended moves.

The first author plans to address these and related issues elsewhere.

\section*{Acknowledgments}
We thank Laura Escobar, Sergey Fomin, Victor Reiner, Steven Sam, Mark Shimozono for helpful remarks. We also made extensive use
of {\tt Macaulay 2} during our investigation. AY was supported by an NSF grant.

\end{document}